\documentclass[11pt]{article}

\usepackage{amssymb, amsmath, amsthm}
\usepackage[colorlinks]{hyperref}

\usepackage{comment}

\theoremstyle{plain}
\newtheorem{lemma}{Lemma}
\newtheorem{theorem}{Theorem}

\newtheorem{corollary}{Corollary}

\theoremstyle{remark}
\newtheorem{remark}{Remark}

\newcommand{\R}{{\mathbb R}}

\newcommand{\Z}{{\mathbb Z}}
\newcommand{\N}{{\mathbb N}}
\renewcommand{\P}{{\mathbb P}}
\newcommand{\ii}{\mathfrak{i}}
\newcommand{\dd}{\mathrm{d}}
\newcommand{\E}{\mathbb{E}}

\newcommand\ind[1]{\mathbb{I}{\left\{#1\right\}}} 
\newcommand{\sgn}{\mathrm{sgn}}

\def\eps{\varepsilon}

\title{Strong renewal theorem and local limit theorem \\
in the absence of regular variation}

\author{P\'eter Kevei\thanks{Bolyai Institute, University of Szeged,
Aradi v\'ertan\'uk tere 1, 6720 Szeged, Hungary; 
e-mail: \texttt{kevei@math.u-szeged.hu}}\quad and\quad
Dalia Terhesiu\thanks{Mathematisch Instituut,
University of Leiden, Niels Bohrweg 1, 2333 CA Leiden, Netherlands;
e-mail: \texttt{daliaterhesiu@gmail.com} }}

\date{}

\begin{document}

\maketitle

\begin{abstract}
We obtain a strong renewal theorem with infinite mean beyond regular variation,
when the underlying distribution belongs to the domain of geometric
partial attraction a semistable law with index $\alpha\in (1/2,1]$.
In the process we obtain local limit theorems  for both finite and infinite mean, that is 
for the whole range 
$\alpha\in (0,2)$. We also derive the asymptotics of the renewal function for $\alpha\in 
(0,1]$.
\end{abstract}

\section{Introduction}

Strong renewal theorems (SRT) with infinite mean that have regularly varying (with 
parameter $\alpha\in [0,1]$)
underlying renewal distributions are nowadays completely understood. 
The SRT in the  one-sided lattice case with $\alpha\in (1/2, 1)$ has been
obtained by Garsia and Lamperti~\cite{GL} and it was  later generalized to the 
nonarithmetic case by Erickson \cite{Erickson70}.
The latter also treats the case $\alpha=1$. As noted in~\cite{GL}, the mere regular 
variation is insufficient in the  
range $\alpha\in (0, 1/2)$. The problem of finding necessary and sufficient conditions 
has 
recently been solved by
Caravenna and Doney~\cite{CaravennaDoney} directly in the two-sided case. 
For more information on improved sufficient conditions for this problematic range we 
refer 
to~\cite{CaravennaDoney}. For a complete treatment of the two-sided $\alpha=1$ case we 
refer to Berger~\cite{Berger}. Very recently, Uchiyama \cite{Uchi} obtained 
asymptotic results for the renewal function for relatively stable variables with
infinite mean. A nonnegative random variable is relatively stable if and only if
its truncated mean is slowly varying. This roughly corresponds to the case 
$\alpha = 1$, but the tail is not necessarily regularly varying. To the best 
of our knowledge, Uchiyama's paper is the only one where the infinite mean 
case in the absence of regular variation is treated.
We also remark that  renewal theory with no moments 
(roughly, the $\alpha=0$ case)
has been dealt with in \cite{AlexanderBerger}.

In this paper we are interested in SRT with infinite mean beyond regular variation. More 
precisely, we focus on distributions in the domain of geometric partial 
attraction of a semistable law. The class of semistable laws,
introduced by Paul L\'evy, is a natural 
extension of stable laws. They are the limits of 
appropriately centered and normed sums
of iid random variables along geometrically increasing subsequences.
Analytically, the tail of the L\'evy measure of the non-Gaussian stable laws are
$x^{-\alpha}$, for some $\alpha \in (0,2)$, while for semistable laws
an additional logarithmically periodic factor appears.
The same logarithmically periodic function appears in the characterization
of the domain of geometric partial attraction. 
A brief background on semistable laws is 
provided in Section~\ref{sec:semistable}.
For definitions, properties, and history of semistable laws 
we refer to Sato \cite[Chapter 13]{Sato},  Megyesi 
\cite{M}, Cs\"org\H{o} and Megyesi \cite{CsM}, and the references therein.  

Our main results on SRT for the case of one-sided $\alpha\in (1/2, 1)$ semistable renewal 
distributions
are Theorem~\ref{thm:SRTGL} (arithmetic case) and  Theorem~\ref{thm:SRTEr} 
(nonarithmetic and nonlattice cases). Unlike in~\cite{GL} and \cite{Erickson70}, 
we cannot use the precise
asymptotic of  the characteristic function. Although the  characteristic function 
asymptotic in Theorem~\ref{thm:phi-asy-gen} is an important ingredient of our 
proofs, the strategy is the systematic use of local limit theorems (LLT).
The LLTs for semistable laws that we obtain here for both finite and infinite mean, that 
is for the whole range 
$\alpha\in (0,2)$, are new. These are 
Theorem~\ref{thm:local-sst} (lattice case) and  
Theorem~\ref{thm:llt-stone} (nonlattice case).

Note that concerning LLT lattice and nonlattice distributions have to be treated
separately, while concerning renewal theorems arithmetic and nonarithmetic
distributions are different. Our proof of the SRT relies on the LLT.
For arithmetic distributions we use the lattice LLT, while for 
nonlattice distributions we use the nonlattice LLT. In the proof of the 
remaining case for nonarithmetic lattice distributions, we use the lattice 
LLT together with the fact that the irrational rotation is uniquely
ergodic, therefore it smooths out the mass at infinity.
In particular, our proof in the nonarithmetic case is different from 
Erickson's \cite{Erickson70} method.

As clarified in~\cite[Section 4.1]{CaravennaDoney} via probabilistic arguments, local 
limit results (namely, LLT and Local Large Deviation) are sufficient to prove SRT for the 
regularly varying case in range $\alpha\in (1/2,1)$.
An analytic proof of this fact is 
absent in the literature.
Our proof of Theorem~\ref{thm:SRTGL} does precisely this while answering the current 
question on SRT in the semistable setting. In the process we show that the proofs 
in \cite{GL} and \cite{Erickson70} can be written using
just the LLT together with a `rough' asymptotics of the characteristic function.

While the characteristic function asymptotics for $\alpha=1$ 
in Theorem~\ref{thm:phi-asy-gen},
are considerably more difficult than for the range $\alpha\in (0, 1)$, 
the proof of the SRT (Theorem \ref{thm:srt-a1}) is in fact simpler, 
and was obtained in a more general setup in \cite{Uchi}.

In Theorem~\ref{thm:Uasy-sem} we obtain the asymptotics of the renewal function for  
$\alpha \in (0,1]$ semistable renewal distributions.
Previous similar, partial results are obtained in Kevei~\cite[Theorem 2.1]{K} and in 
the authors' previous paper \cite[Theorem 2]{KT},
which provide a Karamata type theorem in the absence of strict regular variation. 
The basic observation used in the proof of Theorem~\ref{thm:Uasy-sem} is that the  
semistable limit theorem obtained in \cite{CsM}
in terms of characteristic functions  (but not LLT)  
together with an inversion formula can be used to obtain 
the asymptotics of the renewal function.  
This type of argument is not needed (although it makes sense)  
in the regular variation setting because the Karamata Tauberian
theorem gives the desired result.

All the proofs are gathered together in Section \ref{sect:proofs}.

\section{Characteristic function asymptotics}
\label{sec:char}

Let $X$ be a random variable with distribution 
function $F(x) = \P(X \leq x)$.
Put $\overline F(x) = 1 - F(x)$.
For $r > 1$ introduce the set of logarithmically 
periodic functions
\[
\begin{split}
\mathcal{P}_{r} = \Big\{  p: (0,\infty) \to (0,\infty)  : & \,
\inf_{x \in [1,r]} p(x) > 0, \ 
p \text{ is bounded,} \\ 
& \text{ right-continuous, and }  p(x r) = p(x), \ 
\forall x >0\Big\}.
\end{split}
\]
Assume that for some $r > 1$, $\alpha \in (0,1)$, and a 
slowly varying function $\ell$
\begin{equation} \label{eq:def-rlp}
\lim_{n \to \infty} \frac{(r^n z)^\alpha}{\ell(r^n)} \overline F (r^n z)
= p_0 (z), \quad z \in C_{p_0},
\end{equation}
where the limit $p_0$ is not identically 0.
Then the appearing function $p_0$ is necessarily log-periodic, 
i.e.~$p_0(r x) = p_0(x)$, and since $F$ is monotone, $p_0(x) x^{-\alpha}$
is nonincreasing. Then $\overline F$ is 
called \emph{regularly log-periodic}. 
A stronger assumption is
\begin{equation*} 
\overline F(x) = \ell(x) x^{-\alpha} p_0(x),
\quad \text{ with } p_0 \in \mathcal{P}_r,
\end{equation*}
which follows from \eqref{eq:def-rlp} if $p_0$ is continuous.

Let $U(x) = \sum_{n=0}^\infty F^{*n}(x)$ be the corresponding 
renewal function, where $*n$ stands for the usual convolution power.
If \eqref{eq:def-rlp} holds then
a slight generalization of~\cite[Theorem 2]{KT}  (with the identical proof)
shows that
\[
\lim_{n \to \infty} \frac{U(r^n z) \ell(r^n)}{(r^n z)^\alpha} = p_1(z),
\]
where $p_1$ can be determined explicitly, see \cite[Theorem 2]{KT}.

For finer results we first need the asymptotic behavior of the characteristic function
of $X$. In what follows, oscillatory integrals appear naturally. The notation
$\int_0^{\infty-}$ means that the integral is understood as improper Riemann
integral, and not as Lebesgue integral on $[0,\infty)$.

Assume that
\begin{equation} \label{eq:df-ass}
\begin{split}
& \overline F(x) = \frac{\ell(x)}{x^\alpha} h(x), \\
& F(-x) = \frac{\ell(x)}{x^\alpha} k(x), \ x > 0,
\end{split}
\end{equation}
where $\alpha \in (0,2)$, the function $\ell$ is a slowly varying, and
$h$ and $k$ are either identically 0, or positive bounded functions
with strictly positive infimum, and at least
one of them is not identically zero. Let
\[
\varphi(t) = \E e^{\ii t X} = \int_{\R} e^{\ii t x } \dd F(x).
\]
We write $\Re$ for the real part and $\Im$ for the imaginary part.

\begin{theorem} \label{thm:phi-asy-gen}
Assume that \eqref{eq:df-ass} holds. If $\alpha \in (0,1)$ then
\[
\limsup_{t \to 0}
\frac{|1 - \varphi(t)|}{|t|^\alpha \ell(1/|t|)} < \infty.
\]
Furthermore, if $h(x)x^{-\alpha}$ and $k(x) x^{-\alpha}$ 
in \eqref{eq:df-ass} are ultimately nonincreasing then as $t \to 0$
\[
1 - \varphi(t)  \sim - \ii \sgn(t) \, |t|^\alpha 
\ell(1/|t|) p_2(t),
\]
where 
\[
p_2(t) = 
\int_0^{\infty-} y^{-\alpha} \left[ 
h(y/|t|)  e^{\ii y \sgn(t)} - k(y/|t|) e^{-\ii y \sgn(t)} \right] \dd y.
\]

If $\alpha \in (1,2)$ then as $t \to 0$
\[
1 + \ii t \E X - \varphi(t) \sim 
-\ii \sgn(t) \, |t|^\alpha \ell(1/|t|) p_2(t),
\]
where 
\[
p_2(t) = 
\int_0^{\infty} y^{-\alpha} \left[ 
h(y/|t|)  (e^{\ii y \sgn(t)}-1) - k(y/|t|) 
(e^{-\ii y \sgn(t)}-1) \right] \dd y.
\]

If $\alpha = 1$
\[
\limsup_{t \to 0}
\frac{\Re (1 - \varphi(t))}{|t| \ell(1/|t|)} < \infty
\ \text{ and } \ 
\limsup_{t \to 0}
\frac{| \Im \varphi(t) | }{|t| L(1/|t|)} < \infty,
\]
where
\[
L(x) = \int_1^x \left[ \overline F(u) + F(-u) \right] \, \dd u
\]
is a slowly varying function such that $L(x) / \ell(x) \to \infty$
as $x \to \infty$. In the one-sided case, i.e.~if $k \equiv 0$ then
\[
| \Im \varphi(t) | \sim |t| L(1/|t|),
\]
also holds.
Furthermore if $h(x)/x$ and $k(x)/x$ are ultimately nonincreasing
then
\[
\Re (1 - \varphi(t) ) \sim |t| \ell(1/|t|) 
\int_0^{\infty-}
\frac{\sin y}{y} \left( h(y/|t|) + k(y/|t|) \right) \dd y.
\]

Finally, for any $\alpha \in (0,2)$
\begin{equation*} 
\liminf_{t \to 0}
\frac{\Re (1 - \varphi(t))}{|t|^\alpha \ell(1/|t|)} > 0.
\end{equation*}
\end{theorem}

\begin{remark}
For $\alpha \in (0,1)$
some monotonicity conditions are needed for the finiteness of the improper
integral in $p_2$. Indeed, it is easy to construct examples such that
$\int_0^{\infty -} \ell(x) x^{-\alpha} \cos x \, \dd x$ does not exist
and $\lim_{x \to \infty} \ell(x) = 1$. On the other hand, for $\alpha > 1$
the function $p_2$ is defined as a Lebesgue integral.

We note that the $\alpha = 1$ case is more complicated, as usual.
The main difficulty is that the order of the real and imaginary parts
are different and in general, the imaginary part is larger. However,
for symmetric distributions the imaginary part disappears.
For a treatment of $\alpha = 1$ in the regular variation case we refer to~\cite{AD98}.
See also Lemma 2 by Erickson \cite{Erickson70}, or Pitman \cite{pitman}.
For the corresponding result in the regularly varying case see 
Theorem 2.6.5 in Ibragimov and Linnik \cite{IL}, 
for results on more general integral transform see also 
Theorem 4.1.5 in Bingham et al.~\cite{BGT}.
\end{remark}

Let $X$ be a random variable with distribution
function $F$. Assume that
\begin{equation} \label{eq:def-rlp2-gen}
\begin{gathered}
\overline F(x) = \ell(x) x^{-\alpha} p_R(x), 
\ F(-x) = \widetilde \ell(x) x^{-\alpha} p_L(x), 
\ \ell(x) \sim \widetilde \ell(x), \\
\ell, \widetilde \ell \ \text{slowly varying}, \ 
\alpha \in (0,2), \ p_R, p_L \in \mathcal{P}_r \cup \{ 0 \},
\ p_L + p_R \neq 0.
\end{gathered}
\end{equation}
Notice that, due to the logarithmic periodicity of $p_R$ and $p_L$ the functions
$p_R(x) x^{-\alpha}$ and $p_L(x) x^{-\alpha}$ are both nonincreasing. Therefore
the following is an immediate consequence of Theorem \ref{thm:phi-asy-gen}.

\begin{corollary} \label{cor:phi-asy}
Assume that \eqref{eq:def-rlp2-gen} holds, and if $\E |X| < \infty$ then
$\E X = 0$. Then, for $\alpha \neq 1$, as $t \to 0$
\[
{1 - \varphi(t)} \sim 
- \ii \sgn(t) \, |t|^\alpha \ell(1/|t|) p_2(t), 
\]
where 
\[
p_2(t) = 
\begin{cases}
\int_0^{\infty-} y^{-\alpha} \left[ 
p_R( \frac{y}{|t|} )  e^{\ii y \sgn(t)} - 
p_L( \frac{y}{|t|} ) e^{-\ii y  \sgn(t)} \right]
\dd y, & \alpha < 1, \vspace{3pt} \\
\int_0^\infty 
y^{-\alpha} \left[ 
p_R(\frac{y}{|t|})  (e^{\ii y \sgn(t) } -1)- 
p_L(\frac{y}{|t|}) (e^{-\ii y \sgn(t) } -1) \right] \dd y, & 
\alpha > 1.
\end{cases}
\]
While for $\alpha = 1$
\[
\Re (1 - \varphi(t)) \sim |t| \ell(1/|t|)
\int_0^\infty \frac{\sin y}{y} ( 
p_R(y/|t|) + p_L(y/|t|) ) \,  \dd y.
\]
\end{corollary}

\section{Semistable laws}
\label{sec:semistable}

Semistable laws are limits of centered and normed sums of iid random variables along 
subsequences $k_n$ for which 
\begin{equation} \label{eq:H} 
k_{n}<k_{n+1} \ \text{for } n\geq 1 \ \text{and } 
\lim_{n\to\infty}\frac{k_{n+1}}{k_n}=c > 1
\end{equation}
hold. Since $c=1$ corresponds to the stable case (\cite[Theorem 2]{M}), we assume 
that $c > 1$. 
In what follows we let $c$ be as defined in \eqref{eq:H}.

The characteristic function of a non-Gaussian semistable random variable
$V$ has the form
\begin{equation} \label{eq:def-sscharf} 
\psi(t) =
\E e^{\ii t V} =
\exp \left\{ \ii t a + 
\int_{-\infty}^\infty (e^{\ii t x } - 1 - \ii t x \ind{|x| \leq 1} ) 
\, \Lambda( \dd x ) \right\}, 
\end{equation}
with $\ind{\cdot}$ standing for the indicator function, 
where $a \in \R$, and for
the L\'evy measure $\Lambda$, we have
$\Lambda((x, \infty)) = M_R(x) x^{-\alpha}$, 
$\Lambda((-\infty, -x)) = M_L(x) x^{-\alpha}$, where
$M_R, M_L \in \mathcal{P}_{c^{1/\alpha}} \cup \{ 0 \}$, such that 
not both of them are 0. We further assume that $V$ is nonstable, 
that is either $M_R$ or $M_L$ is not constant.

In the following $X, X_1, X_2, \ldots $ are iid random variables with distribution 
function $F(x) = \P ( X \leq x )$. Let $S_n = X_1 + \ldots +X_n$ denote the 
partial sum.
We fix a semistable random variable $V=V(R, M)$ with distribution function $G$ and
characteristic function $\psi$ in \eqref{eq:def-sscharf}.
The random variable $X$ belongs to the domain of 
geometric partial attraction of the semistable law $G$ if there is a 
subsequence $k_n$ for which \eqref{eq:H} holds, and a norming and  a centering sequence 
$A_n, C_n$, such 
that
\begin{equation} \label{eq:dgp-conv}
\frac{\sum_{i=1}^{k_n} X_i  - C_{k_n}}{A_{k_n}} \rightarrow^d V,
\end{equation}
where $\rightarrow^d$ means convergence in distribution.
By \cite[Theorem 3]{M}, without loss of generality we may assume that 
\begin{equation} \label{eq:An}
A_n = n^{1/\alpha} \ell_1(n), \quad 
C_n = n \int_{1/n}^{1-1/n} Q(s) \, \dd s,
\end{equation}
with some slowly varying function $\ell_1$,
where $Q(s) = \inf \{ x: \, F(x) \geq s \}$, $s \in (0,1)$ is the
quantile function of $F$.

In order to characterize the domain of geometric partial attraction
we need some further 
definitions. As $k_{n+1} / k_n \to c > 1$, for any 
$x$ large enough there is a unique $k_n$ such that 
$A_{k_n} \leq x < A_{k_{n+1}}$. Define
\begin{equation*} 
\delta(x) = \frac{x}{A_{k_n}}.
\end{equation*}
Note that the definition of $\delta$ does depend on the norming sequence.
Finally, let  
\begin{equation*} 
x^{-\alpha} \ell(x) = \sup \{ t : t^{-1/\alpha} \ell_1(1/t) > x \}.
\end{equation*}
Then $A_x = A(x) = x^{1/\alpha} \ell_1(x)$ and $ B(y) = y^\alpha / \ell(y)$ are 
asymptotic 
inverses of each 
other, i.e.
\begin{equation} \label{eq:AB}
A(B(x)) \sim B(A(x)) \sim x \quad \text{as } \, x \to \infty, 
\end{equation}
and
$x^{1/\alpha} \ell_1(x) 
\sim \inf \{ y : x^{-1} \geq y^{-\alpha} \ell(y) \}$. 
Thus $\ell$ and $\ell_1$ asymptotically determines each other.
For properties of asymptotic inverse of regularly varying functions we refer to
\cite[Section 1.7]{BGT}.

By Corollary 3 in \cite{M} \eqref{eq:dgp-conv} holds on the 
subsequence $k_n$ with norming sequence $A_{k_n}$ if and only if 
\begin{equation} \label{eq:dgp-distf}
\begin{split}
& \overline F(x) =  \frac{\ell(x)}{x^\alpha}
 [ M_R(\delta(x)) + h_R(x) ], \\
& F(-x) = \frac{\ell(x)}{x^\alpha} [ M_L(\delta(x)) + h_L(x)],
\end{split}
\end{equation}
where $h_R, h_L$ are right-continuous functions such that
$\lim_{n \to \infty} h_{R/L}(A_{k_n} x ) = 0$, whenever $x$ is a 
continuity point of $M_{R/L}$.
Moreover, if $M_{R/L}$ is continuous, then $\lim_{x \to \infty} h_{R/L}(x) = 0$.

Clearly, \eqref{eq:dgp-distf} implies \eqref{eq:df-ass}. 
Thus if $F$ belongs to the domain of geometric partial attraction of a semistable
law, then Theorem \ref{thm:phi-asy-gen} applies.

Conditions \eqref{eq:def-rlp2-gen} and \eqref{eq:dgp-distf} are similar, but
the $\delta$ function in \eqref{eq:dgp-distf} complicates the 
asymptotics. In the special case $\ell_1 \equiv 1$ and
$k_n = \lfloor c^n \rfloor$, the function $\delta(x)$ can be replaced by
$x$ in \eqref{eq:dgp-distf}. 
Then \eqref{eq:def-rlp2-gen} with $\ell \sim 1$ 
is equivalent to \eqref{eq:dgp-distf}
with $h_{R/L}(x) \to 0$ as $x \to \infty$. In general, 
\eqref{eq:def-rlp2-gen} is a stronger condition.

\begin{lemma} \label{lemma:ss-logper}
Assume \eqref{eq:def-rlp2-gen}. Then there exists a subsequence $(k_n)$
satisfying \eqref{eq:H} with $c = r^\alpha$ such that \eqref{eq:dgp-distf}
holds with $M_R = p_R$ and $M_L = p_L$.
\end{lemma}

\begin{proof}
Recall the definition of $A$ and $B$. Define $k_n = B(c^{n/\alpha})$.
For notational ease we suppress the integer part.
Since $B$ is regularly varying with index $\alpha$, condition \eqref{eq:H}
holds. By \eqref{eq:AB} we have $A_{k_n} \sim c^{n/\alpha}$. Writing
\[
\overline F(x) = \frac{\ell(x)}{x^\alpha} \left[ p_R(\delta (x)) +
(p_R(x) - p_R(\delta(x))) \right],
\]
we only have to show that $\lim_{n \to \infty} h_R(A_{k_n} x) = 0$
holds whenever $x$ is a continuity point of $p_R$, for
$h_R(x) = p_R(x) - p_R(\delta(x))$. For simplicity fix
$x \in (1,c^{1/\alpha})$ to be a continuity point of $p_R$. 
Then $A_{k_n} \leq A_{k_n} x < A_{k_n +1}$
for large $n$, thus $\delta (A_{k_n} x) = A_{k_n} x / A_{k_n} = x$.
On the other hand, by the logarithmic periodicity of $p_R$
\[
p_R(A_{k_n} x ) = p_R ( c^{-n/\alpha} A_{k_n} x ) \to p_R(x),
\]
which implies that $h_R(A_{k_n} x) \to 0$. Clearly, the same argument
works for $F(-x)$.
\end{proof}

It is easy to give examples that show that the converse is not true.
Choose $\alpha = 1$, $c = 2$, $\ell(x) = \ell_1(x) = \log_2 x$, 
$k_n = 2^n$, $p_R = 2^{\{ \log_2 x \}}$, $p_L \equiv 0$, where
$\log_2$ stands for the base-2 logarithm, and $\{ \cdot \}$ is the
fractional part. Define for $x > 3$
\[
\overline F(x) = 2^{-\lfloor \log_2 x - \log_2 \log_2 x \rfloor} =
\frac{\log_2 x}{x} 2^{\{ \log_2 x - \log_2 \log_2 x\}}.
\]
Some lengthy but straightforward calculation shows that
\eqref{eq:dgp-distf} holds, but \eqref{eq:def-rlp2-gen} does not.

\medskip

For $x > 0$ (large) we define the position parameter as
\begin{equation} \label{eq:def-gamma}
\gamma_x  = \gamma(x) =\frac{x}{k_n}, \quad \text{where } k_{n-1} <  x \leq k_n.
\end{equation}
We say $u_n$ \emph{circularly converges to $u \in (c^{-1}, 1]$}, $u_n \stackrel{cir}{\to} 
u$, if
$u \in (c^{-1}, 1)$ and $u_n \to u$ in the usual sense, 
or $u=1$ and $(u_n)$ has limit points $c^{-1}$, or $1$, or both.
From Theorem 1 \cite{CsM} we see that (\ref{eq:dgp-conv}) holds along a subsequence 
$(n_r)_{r=1}^\infty$ (instead of $k_n$) if and only if 
$\gamma_{n_r} \stackrel{cir}{\to} \lambda \in (c^{-1}, 1]$ as $r \to \infty$.
In this case, by \cite[Theorem 1]{CsM} (or directly from the relation
$- R_\lambda(x) = \lim_{r \to \infty} n_r \overline F(A_{n_r} x)$) the 
L\'evy measure of the limit
\begin{equation*} 
\begin{split}
& \Lambda_\lambda((x,\infty))  = x^{-\alpha} {M_R(\lambda^{1/\alpha} x)} \\
& \Lambda_\lambda((-\infty, -x))  = x^{-\alpha} M_L(\lambda^{1/\alpha} x),
\quad x > 0. 
\end{split}
\end{equation*}
For any $\lambda > 0$ 
let $V_\lambda$ be a semistable random variable with characteristic 
and distribution function 
\begin{equation}  \label{eq:semistable-chf-df-lambda}
\begin{split}
\psi_\lambda(t) &  =
\E e^{\ii t V_\lambda} =
\exp \left\{  \ii t a_\lambda + 
\int_{-\infty}^\infty \left(e^{\ii t x } - 1
- \ii t x \ind{|x| \leq 1 } \right) 
\Lambda_\lambda( \mathrm{d}  x) \right\} \\
G_\lambda(x) & = \P ( V_\lambda \leq x),
\end{split}
\end{equation}
where $a_\lambda \in \R$, for its precise form see \cite[Theorem 1]{CsM}.
Thus, whenever $\gamma_{n_r} \stackrel{cir}{\to} \lambda$,
\begin{equation*} 
\frac{\sum_{i=1}^{n_r} X_i -C_{n_r}} {A_{n_r}} \to^d V_{\lambda}
\quad \text{as } r \to \infty.
\end{equation*}
To ease notation we define $\Lambda_\lambda$, $G_\lambda$ for any $\lambda > 0$, but note
that $\Lambda_{c \lambda} \equiv \Lambda_{\lambda}$, $G_{c \lambda} \equiv G_\lambda$, so 
these functions, distributions are different for $\lambda \in (c^{-1}, 1]$.

Let $X, X_1, X_2, \ldots$ be iid random variables with
distribution function $F$ such that \eqref{eq:dgp-distf} holds.
Cs\"org\H{o} and Megyesi \cite[Theorem 2]{CsM} showed the following merging result:
\begin{equation} \label{eq:merge}
\lim_{n \to \infty} \sup_{x \in \R}
\left| 
\P \left( \frac{S_n -C_n}{A_n} \leq x \right) -
G_{\gamma_n}(x) \right| = 0.
\end{equation}
The main theorem in \cite{Cs07} implies that $G_\lambda$ is $C^\infty$,
in particular its density function $g_\lambda$ exists.

\section{Local limit theorems for semistable laws}
\label{sec:LLTlatnonlat}

We prove local limit theorems for the distributions in the domain
of geometric partial attraction of semistable laws. 
As usual we have to distinguish between lattice and nonlattice
distributions. We first consider the lattice case.

A random variable, or its distribution is called lattice, 
if it is concentrated on the set $\{ a + h \Z \}$ for some 
$a \in \R$ and $h > 0$. If $a = 0$ the distribution is called 
arithmetic, or centered lattice.
The largest possible $h$ is the span of the lattice
distribution. We assume that $a = 0$ and $h = 1$, i.e.~the 
distribution is integer valued with span 1.
We prove the analogue of
Gnedenko's Local Limit Theorem (\cite[Theorem 8.4.1]{BGT}, 
\cite[Theorem 4.2.1]{IL}). The statement can be readily extended to
the general lattice case.

\begin{theorem} \label{thm:local-sst}
Let $X, X_1, \ldots$ be integer valued iid random variables 
with span 1, such that \eqref{eq:dgp-distf} holds. Then
\[
\lim_{n \to \infty}
\sup_{k} | A_n \P ( S_n = k) - g_{\gamma_n}((k-C_n)/A_n )| = 0.
\]
\end{theorem}

The Fourier analytic proof relies on the inversion formula
\begin{equation} \label{eq:latt-inv}
\P ( S_n = k) 
= \frac{1}{2 \pi} \int_{-\pi}^\pi e^{- \ii t k} \varphi(t)^n 
\, \dd t,
\end{equation}
and on the merging result \eqref{eq:merge}.

In the nonlattice case we extend Stone's local limit theorem \cite{Stone},
see also \cite[Theorem 8.4.2]{BGT}.

\begin{theorem} \label{thm:llt-stone}
Let $X, X_1, \ldots$ be iid nonlattice random variables
such that \eqref{eq:dgp-distf} holds. Then for any $h > 0$
\[
\lim_{n \to \infty}
\sup_{x} \left| \frac{A_n}{2h} \P ( S_n \in (x-h, x+h] ) - 
g_{\gamma_n}((x-C_n)/A_n )
\right| = 0.
\]
\end{theorem}

The difficulty in the nonlattice setup is the lack of a simple 
inversion formula as \eqref{eq:latt-inv}. Instead, in the usual
Fourier inversion formula one has to take limits. The standard
trick to overcome this is to add a small continuous random variable
with compactly supported characteristic function.
Fix $T > 0$ and let $Y$ be a random variable with density and 
characteristic function 
\begin{equation} \label{eq:def-Y}
j(x) = \frac{1 - \cos (Tx)}{\pi T x^2}, \quad 
\eta(t) = 
\begin{cases}
1 - \frac{|t|}{T}, & \text{for } t \in [-T, T], \\
0, & \text{otherwise}.
\end{cases}
\end{equation}
Then the inversion formula gives
\begin{equation} \label{eq-inversion}
\P ( S_n + Y \in (x- h, x+h] ) = \frac{h}{\pi} 
\int_{-T}^T \frac{\sin th}{th}
e^{-\ii tx} \varphi^n(t) \left( 1 - \frac{|t|}{T} \right) \dd t.
\end{equation}
Having this formula the proof goes as in the lattice case, only
at the end we have to get rid of the small perturbation.

\section{Strong renewal theorem in the semistable setting}

In what follows, we consider only nonnegative  random variables
with infinite mean in the domain of geometric partial attraction
of a semistable law. 
In particular, $\alpha \in (0,1]$. For $\alpha \in (0,1)$
there is no need for centering, i.e.~in \eqref{eq:An} we choose
$C_n \equiv 0$.

Using the local limit theorems, we obtain the analogue of~\cite[Theorem 1.1]{GL}
in the semistable setting, that is  assuming~\eqref{eq:dgp-distf}.  
Unlike in~\cite{GL}, we cannot use the precise
asymptotic of  $(1-\varphi(t))^{-1}$. Instead, we  heavily exploit the LLT, namely 
Theorems~\ref{thm:local-sst} and \ref{thm:llt-stone}
together with the asymptotic of $(1-\varphi(t))^{-1}$ 
obtained in Theorem~\ref{thm:phi-asy-gen}.

We start with the arithmetic case, and assume that $X$ is integer valued 
with span 1.
With the same notation as in~\cite{GL} introduce the 
renewal sequence 
\begin{equation} \label{eq:un} 
\begin{split}
u_n  =\sum_{k=0}^\infty \P(S_k=n) 
 = \frac{1}{\pi}\Re\int_0^\pi(1-\varphi(t))^{-1}\, e^{-\ii nt}\, \dd t,
\end{split}
\end{equation}
where we used the inversion \eqref{eq:latt-inv}.

\begin{theorem}
\label{thm:SRTGL}
Assume that $X$ is a nonnegative integer valued random variable
with span 1 and 
\eqref{eq:dgp-distf} holds with $\alpha\in (1/2,1)$. 
Set $B(x)= x^\alpha\ell(x)^{-1}$. Then
\begin{equation*}
\lim_{n\to\infty}
\Big| n^{1-\alpha}\ell(n)u_n- 
\alpha
\int_0^\infty g_{\gamma(B(n) x^{-\alpha})}(x) 
\, x^{-\alpha}\, \dd x \Big| =0.
\end{equation*}
\end{theorem}

The estimate of the main term above holds in the whole range
$\alpha \in (0,1)$, and it is treated separately in the following
statement. It is the analogue of Lemma 2.2.1 in \cite{GL}.

\begin{lemma}\label{lemma-riemsum}
Assume that $X$ is a nonnegative integer valued random variable
with span 1 and \eqref{eq:dgp-distf} holds with $\alpha\in (0,1)$. 
For any  $L > 1$
\[
\begin{split}
\limsup_{n \to \infty}
\Big| n^{1-\alpha} \ell(n)
\sum_{k=B(n/L^2)}^{B(nL)} 
\P(S_k = n) 
- \alpha \int_{L^{-1}}^{L^2} 
g_{\gamma(B(n) y^{-\alpha})}(y) y^{-\alpha} \, \dd y
\Big|
\leq  L^{-1}.
\end{split}
\]
\end{lemma}

Recall that the renewal function is denoted by 
$U(y):=\sum_{n}F^{n*}(y)$.  The next result gives the SRT in the 
semistable  nonarithmetic case.

\begin{theorem}
\label{thm:SRTEr}
Assume that $X$ is a nonnegative nonarithmetic random variable and 
\eqref{eq:dgp-distf} holds with $\alpha\in (1/2,1)$. 
Set $B(x)= x^\alpha\ell(x)^{-1}$. Then for any $h>0$,
\begin{equation*}
\lim_{y\to\infty}
\Big| \frac{y^{1-\alpha}\ell(y)}{2h}\left(U(y+h)-U(y-h)\right)
-\alpha
\int_0^\infty g_{\gamma(B(y) x^{-\alpha})}(x) 
\, x^{-\alpha}\, \dd x \Big| =0.
\end{equation*}
\end{theorem}

In the nonarithmetic lattice case without loss of generality we 
assume that $X$ has span 1, and that $X \in a + \N$, where 
$a \in (0,1)$ is irrational.
The proof of Theorem~\ref{thm:SRTEr} in this case
is essentially the same as the proof
of Theorem~\ref{thm:SRTGL}, except for the treatment of the leading term.
To make this precise  we introduce the following notation.

Let $\widetilde X = X - a$ denote the centered version of $X$,
and $\widetilde S_k = S_k  - k a$. Fix $0 < h < 1/2$.
Define
\[
I_{k,y} =
\begin{cases}
1, & \text{if $(y-ka - h, y-ka+h]$ contains an integer}, \\
0, & \text{otherwise},
\end{cases}
\]
and let $\langle y - ka \rangle $ denote the unique integer in
the interval $(y-ka-h, y-ka+h]$ if $I_{k,y} = 1$, and 0 otherwise.
Then
\[
\P ( S_k \in (y-h, y+h ] ) =
\P ( \widetilde S_k \in ( y - ka - h, y -ka + h] )
= I_{k,y} \P ( \widetilde S_k = \langle y-ka \rangle )
\]
and 
\begin{equation}
 \label{eq:u-nonar}
U(y+h) - U(y-h) = \sum_{k=0}^\infty \P ( S_k \in (y-h, y+h])
= \sum_{k=0}^\infty I_{k,y} \P( \widetilde S_k = \langle y - ka \rangle ).
\end{equation}

\begin{lemma}\label{lemma-sumnonar}
Assume that $X$ is a nonnegative,  nonarithmetic and lattice
with span 1. Suppose  that \eqref{eq:dgp-distf} holds with $\alpha\in (0,1)$. 
Then for any $h \in (0,1/2)$
\[
\begin{split}
\limsup_{y \to \infty}
\Big| y^{1-\alpha} \ell(y)
\sum_{k=B(y/L^2)}^{B(yL)} 
 I_{k,y} \P( \widetilde S_k = \langle y - ka \rangle )
- 2h \alpha \int_{L^{-1}}^{L^2} 
g_{\gamma(B(y) x^{-\alpha})}(x) x^{-\alpha} \, \dd x
\Big|
\leq  L^{-1}.
\end{split}
\]
\end{lemma}

In the proof of the nonlattice case of  Theorem~\ref{thm:SRTEr} we first apply
the ideas of the arithmetic case to 
the smoothed version as in \eqref{eq-inversion}, 
then `unsmooth' the limit.
\smallskip

The case $\alpha = 1$ is different, already in the regularly varying 
framework. However, the difference is more apparent in the semistable setup,
since the usual limit result holds.
We assume that $\E X = \infty$, because if it was finite, the classical
renewal theorem would work. Our results are special cases
of Lemma 67 in Uchiyama \cite{Uchi2} and Corollary 2 in \cite{Uchi}.
For completeness, we state the results.

Now, instead of \eqref{eq:un} we use the inversion formula 
\begin{equation} \label{eq:u-repGL}
u_n = \frac{2}{\pi} \int_0^\pi 
W(t) \, \cos nt \, \dd t,
\end{equation}
see Lemma 3.1.1 in \cite{GL} or (2.5) in \cite{Erickson70},
where 
\begin{equation*} 
W(t) = \Re \frac{1}{1 - \varphi(t)}
= \frac{\Re (1-\varphi(t))}
{| 1 - \varphi(t) |^2}.
\end{equation*}
The expectation `almost exists' in the sense that the truncated first
moment
\[
L(x) = \int_1^x \overline F(u) \dd u
\]
is slowly  varying. The key ingredient is Lemma 1 in \cite{Uchi}, 
the slow variation of the integral of $W$. 
The regularly varying version is Lemma 3 in \cite{Erickson70}.

\begin{lemma}[Lemma 1 in \cite{Uchi}] \label{lemma:a1-L}
Assume that $X$ is a nonnegative random variable such that
\eqref{eq:dgp-distf} holds with $\alpha = 1$, and 
$\E X = \infty$. Then as $x \to \infty$
\[
\int_0^{1/x} W(t) \, \dd t \sim L(x)^{-1} \frac{\pi}{2}.
\]
\end{lemma}

The arithmetic version of the next result is a special case 
of Lemma 67 in \cite{Uchi2}, and the nonarithmetic version
is a special case of Corollary 2 in \cite{Uchi}.
The proof is based on Lemma \ref{lemma:a1-L} and on the argument 
in \cite{Erickson70}.

\begin{theorem} \label{thm:srt-a1}
Assume that $X$ is a nonnegative random variable such that 
\eqref{eq:dgp-distf} holds with $\alpha = 1$, and 
$\E X = \infty$.
If $X$ is integer valued with span 1 then
\[
\lim_{n \to \infty}  L(n) u_n  = 1, 
\]
while if $X$ is nonarithmetic then for any $h > 0$
\[
\lim_{y \to \infty} 
L(y) ( U(y+h) - U(y-h) )  = 2h.
\]
\end{theorem}

\section{Renewal function in the semistable setting}

In this section we determine the asymptotic of 
$U(y)$, as $y\to\infty$ for any $\alpha \in (0,1)$.
This time we will not exploit the LLT, but simply the merging 
result~\eqref{eq:mergechar} in terms of the characteristic function.
In short, the basic observation is that the semistable limit theorem,
equivalently  the merging result~\eqref{eq:mergechar},
is the only thing one needs to obtain the asymptotic of  
$U(y)$ for both arithmetic and nonarithmetic semistable distributions.
This type of argument is not needed (although it makes sense)  
to obtain the asymptotic of $U(y)$ in the regularly varying (stable)
setting where Karamata's Tauberian theorem gives immediate results.

Recall that $G_{\gamma_k}$ is the semistable distribution 
defined in \eqref{eq:semistable-chf-df-lambda}.
We note that
\begin{align*} 
\int_{0}^{\infty}
G_{\gamma(B(y) x^{-\alpha})} ( x) x^{-\alpha -1} \,
\dd x <\infty.
\end{align*}
At $\infty$ this is clear, while at 0 this follows from
the fact that $G_\gamma(x)$ is exponentially small around 0,
see Theorem 1 by Bingham \cite{Bingham2} (or Lemma 2 in \cite{KT}).

\begin{theorem}
\label{thm:Uasy-sem}
Assume that $X$ is a nonnegative random variable and 
\eqref{eq:dgp-distf} holds with $\alpha\in (0,1)$.
Set $B(x)= x^\alpha\ell(x)^{-1}$. Then
\begin{equation*}
\lim_{y\to\infty}
\Big| y^{-\alpha}\ell(y) U(y)
- \alpha \int_{0}^{\infty}
G_{\gamma(B(y) x^{-\alpha})} ( x) x^{-\alpha-1} \, 
\dd x \Big| =0.
\end{equation*}
\end{theorem}

As a consequence of Theorem \ref{thm:srt-a1}  
we obtain for $\alpha = 1$ the following.

\begin{corollary}
Assume that $X$ is a nonnegative random variable such that 
\eqref{eq:dgp-distf} holds with $\alpha=1$ and $\E X = \infty$.
Then, as $y \to \infty$
\[
U(y) \sim \frac{y}{L(y)}.
\]
\end{corollary}

It is natural to expect 
that under some additional assumption the SRT in Theorems \ref{thm:SRTGL} and 
\ref{thm:SRTEr} remains true for $\alpha \in (0, 1/2]$. 
The problem to find the necessary and sufficient conditions for the SRT
in the regularly varying setup was open for more than 50 years, and  was solved 
recently by Caravenna and Doney \cite{CaravennaDoney}.
In the regularly varying setup, already 
in the first papers \cite{GL, Erickson70} it was pointed out that 
for $\alpha \in (0, 1/2]$ the results hold with $\liminf$ instead of $\lim$,
moreover the exceptional set is negligible in the sense that has density 0.

We do not know what happens for $\alpha \in (0,1/2]$. We only point out 
the essential difficulty to obtain further asymptotics.
By Lemma \ref{lemma-riemsum} for any $\alpha \in (0,1)$
\begin{equation} \label{eq:liminf}
\liminf_{n \to \infty} 
\left[ n^{1-\alpha } \ell(n) \, u_n 
- \alpha \int_0^\infty g_{\gamma(B(n) y^{-\alpha})}(y) y^{-\alpha} 
\, \dd y \right] \geq 0.
\end{equation}
In the regularly varying case \eqref{eq:liminf}, 
together with Theorem \ref{thm:Uasy-sem},
is enough to conclude that for $\alpha \in (0,1/2]$
the liminf in \eqref{eq:liminf} is 0,
moreover the limit exists and equals 0 except in a set of density
0; see \cite[Theorem 1.1]{GL}, \cite[Theorem 2]{Erickson70}, or
\cite[Theorem 8.6.6]{BGT}. If $G$ is any distribution function of
a nonnegative random variable with density $g$, then simply
\[
\alpha \int_0^\infty G(x) x^{-\alpha - 1} \, \dd x =
\int_0^\infty g(x) x^{-\alpha } \, \dd x.
\]
In our case the distribution function itself depends on
$x$, thus the argument above does not work.

\smallskip

\section{Proofs} \label{sect:proofs}

\subsection{Proof of Theorem \ref{thm:phi-asy-gen}}

\textbf{Case 1:} $\alpha \in (0,1)$.
Integration by parts shows
\begin{equation} \label{eq:phi-form}
\begin{split}
& 1 - \varphi(t) = \int_{[0,\infty)} (e^{\ii t x} -1 )
\, \dd \overline F(x)
+ \int_{(0,\infty)} (e^{-\ii t x} -1 ) \, \dd F(-x) \\
& = - \ii t 
\left( 
\int_{0}^{\infty -} \overline F(x) e^{\ii t x} \, \dd x 
- \int_{0}^{\infty -} F(-x) e^{-\ii t x} \, \dd x \right) \\
& = - \ii \sgn (t)  |t| ^{\alpha} 
\int_0^{\infty-}  \ell( \frac{y}{|t|} )  y^{-\alpha} 
\left(  h( \frac{y}{|t|} ) e^{\ii \sgn (t) y}
- k( \frac{y}{|t|}) e^{-\ii \sgn(t) y} \right)  \dd y.
\end{split}
\end{equation}
To ease notation we write $x = |t|^{-1}$. We consider the first term 
in the integral above, and assume $t > 0$.
For any $0 < a < b < \infty$ by the uniform convergence theorem for slowly
varying functions as $x \to \infty$
\begin{equation*} 
\frac{1}{\ell(x)}
\int_a^b h(yx) \ell(y x) y^{-\alpha } e^{\ii y} \dd y 
- \int_a^b h(yx) y^{-\alpha} e^{\ii y } \dd y \to 0.
\end{equation*}
Next we show that the contribution of the integral
on $(0,a)$, and on $(b, \infty)$ is negligible. Indeed, by Karamata's theorem
\begin{equation} \label{eq:0-aint}
\begin{split}
\left| \int_0^a h(yx) \ell(yx) y^{-\alpha } e^{\ii y} \dd y \right|
& \leq C \, x^{\alpha-1} \int_0^{ax} \ell(u) u^{-\alpha} \, \dd u \\
& \sim C \, a^{1-\alpha} \ell(x) \quad \text{as } x \to \infty.
\end{split}
\end{equation}
In the following $C > 0$ is always a finite positive constant, which may be 
different from line to line, and its actual value is not important for us.
On $(b, \infty)$ we consider only the real part. 
Since the function $\overline F(x) = \ell(x) h(x) x^{-\alpha}$
is nonincreasing, by the second mean value theorem for definite
integrals  we obtain
\begin{equation} \label{eq:b-infint}
\begin{split}
\left| \int_{b}^{\infty -}
h(yx) \ell(yx) y^{-\alpha } \cos y \, \dd y \right| 
&\leq h(bx) \ell(bx) b^{-\alpha} \, 
\sup_{z > b} \left| \int_b^z \cos y \, \dd y \right| \\
& \leq C \, \ell(x) b^{-\alpha}.
\end{split}
\end{equation}
Clearly, the inequalities \eqref{eq:0-aint} and \eqref{eq:b-infint}
hold true for the second term in \eqref{eq:phi-form}, therefore
\[
\left| \frac{1-\varphi(t)}{|t|^\alpha \ell(1/|t|)} \right| \leq 
C \left( 
a^{1-\alpha} +  b^{-\alpha} + \int_a^b y^{-\alpha} \, \dd y \right),
\]
showing the first part of the theorem. 
\smallskip

For the more precise asymptotic
first note that with the extra monotonicity condition the function
$p_2$ is well-defined. This follows from the
Leibniz criterion for the finiteness of an alternating series, 
recalling the fact that $h(y) y^{-\alpha}$ and $k(y) y^{-\alpha}$
are ultimately nonincreasing. Moreover, the inequalities \eqref{eq:0-aint} and 
\eqref{eq:b-infint}
hold true  with $\ell(x) \equiv 1$. Therefore
\[
\begin{split}
& \left| \frac{1}{\ell(x)} \int_0^{\infty-} h(xy) \ell(xy) y^{-\alpha}
e^{\ii y} \, \dd y - 
\int_0^{\infty-} h(xy) y^{-\alpha} e^{\ii y} \, \dd y \right| \\
& \leq C (a^{1- \alpha} + b^{-\alpha}) + 
\left| \frac{1}{\ell(x)} \int_a^b h(yx) \ell(yx) y^{-\alpha } 
e^{\ii y} \, \dd y 
- \int_a^b h(yx) y^{-\alpha } e^{\ii y} \, \dd y
\right|,
\end{split}
\]
and the statement follows by letting
$x = 1/|t| \to \infty$, then $a \to 0$ and $b \to \infty$.

\smallskip
\textbf{Case 2:} $\alpha \in (1, 2)$. 
In this case $\E X $ exists, and by 
subtracting, and using that $\E e^{\ii t X} = 1 + \ii t \E X + o(t)$
as $t \downarrow 0$, we may and do assume that $\E X = 0$.
Similarly as in \eqref{eq:phi-form}
\begin{equation*} 
\begin{split}
& 1 - \varphi(t)  
= \int_{\R} \left( 1 -e^{\ii t x} + \ii t x \right) \dd F(x) \\
& = - \ii \sgn(t) |t|^\alpha \int_0^\infty 
\frac{\ell(y/|t|)}{y^\alpha}
\left( 
(e^{\ii \sgn(t) y } - 1) h(y/t) - 
(e^{-\ii \sgn(t) y } - 1) k(y/t) \right) \dd y.
\end{split}
\end{equation*}
As above, 
for any $0 < a < b < \infty$ as $x = |t|^{-1} \to \infty$
\begin{equation*} 
\frac{1}{\ell(x)}
\int_a^b h(yx) \ell(y x) y^{-\alpha } (e^{\ii y} -1) \, \dd y 
- \int_a^b h(yx) y^{-\alpha} (e^{\ii y } - 1) \, \dd y \to 0.
\end{equation*}
Next we show that the contribution of the integral
on $(0,a)$ and on $(b, \infty)$ is negligible.
For $y$ small enough $e^{\ii y} - 1 \sim \ii y$, thus
 by Karamata's theorem
\begin{equation} \label{eq:0-aint>1}
\begin{split}
\left| \int_0^a h(yx) \ell(yx) y^{-\alpha } (e^{\ii y} - 1) \, \dd y \right|
& \leq C \, \int_0^a \ell(yx) y^{1 -\alpha } \, \dd y \\
& \sim C \, a^{2-\alpha} \ell(x) \quad \text{as } x \to \infty.
\end{split}
\end{equation}
Similarly, on $(b, \infty)$ we have
\begin{equation} \label{eq:b-infint>1}
\begin{split}
\left| \int_{b}^{\infty}
h(yx) \ell(yx) y^{-\alpha } (e^{\ii y} -1)  \, \dd y \right| 
& \leq C \ell(x) b^{1-\alpha}.
\end{split}
\end{equation}
Since the inequalities \eqref{eq:0-aint>1} and \eqref{eq:b-infint>1}
hold with $\ell(x) \equiv 1$, therefore
\[
\begin{split}
& \left| 
\frac{1}{\ell(x)} \int_0^\infty y^{-\alpha}  h(xy) \ell(xy) 
(e^{\ii y} - 1) \, \dd  y 
- \int_0^\infty y^{-\alpha} h(xy) (e^{\ii y} - 1) \, \dd y
\right| \\
&  \leq 
C \left( 
a^{2-\alpha} +  b^{1-\alpha} \right)
+ \left| \frac{1}{\ell(x)} \int_a^b h(yx) \ell(yx) 
y^{-\alpha } e^{\ii y} \, \dd y 
- \int_a^b h(yx) y^{-\alpha } e^{\ii y} \, \dd y
\right|,
\end{split}
\]
and statement follows by letting 
$x = 1/|t| \to \infty$, then $a \to 0$ and $b \to \infty$.

\smallskip 
\textbf{Case 3:} $\alpha =1$. In this case the calculations are
more troublesome. Using that 
\[
\int_{(-1,1]} x \, \dd F(x) = \int_{0}^1 
[ \overline F(x) - F(-x) ] \, \dd x
- \overline F(1) + F(-1)
\]
and that $e^{\ii t x} - 1 -\ii t x = O(t^2)$ for $x \in [-1,1]$,
straightforward calculation shows
\begin{equation} \label{eq:a1-phi}
\begin{split}
& 1 - \varphi (t) 
= \int_{\R} ( 1 - e^{\ii t x} ) \, \dd F(x)  \\
& = \! - \ii t \int_{1}^{\infty-} \hspace{-6pt}
\left( \overline F(x) e^{\ii t x} \! 
-  \! F(-x) e^{-\ii tx} \right) \! \dd x \! -
\ii t \hspace{-4pt}
\int_0^1 \! [ \overline F(x) \! - F(-x) ]  \dd x + O(t^2) \\
& = - \ii \sgn(t) |t| \int_{|t|}^{\infty-}
\frac{\ell(y/|t|)}{y} 
\left[ h (y/|t|) e^{\ii \sgn(t) y} 
- k(y/|t|) e^{- \ii \sgn(t) y} \right] \dd y \\
& \quad - \ii t \int_0^1 [\overline F(x) - F(-x)] \, \dd x + O(t^2). 
\end{split}
\end{equation}
In this case
the order of the real and imaginary parts are different. 
As $\sin y \sim y$ at 0, using 
the arguments in \eqref{eq:0-aint} and \eqref{eq:b-infint} 
we have
\begin{equation*} 
\left| 
\frac{1}{\ell(1/|t|)} \int_{|t|}^{\infty-}
\frac{\sin y}{y} \ell(y/|t|)
h(y/|t|) \, \dd y - 
\int_0^b \frac{\sin y}{y} h(y/|t|) \, \dd y \right| 
\leq C b^{-1},
\end{equation*} 
for $t$ small enough, for some $C > 0$. 
Moreover, if $h(y) y^{-1}$ is ultimately monotone
this can be strengthened to
\begin{equation*} 
\left| 
\frac{1}{\ell(1/|t|)} \int_{|t|}^{\infty-} \frac{\sin y}{y} \ell(y/|t|)
h(y/|t|) \, \dd y - \int_0^{\infty-} 
\frac{\sin y}{y} h(y/|t|) \, \dd y \right| 
\to 0
\end{equation*}
as $t \to 0$. Thus the statement for the real part follows.

For the imaginary part in \eqref{eq:a1-phi}
we obtain as in \eqref{eq:b-infint}
\[
\left| 
\int_1^{\infty-} \frac{\cos y}{y} \ell(y/|t|) h(y/|t|) \,
\dd y \right|
\leq C \ell(1/|t|),
\]
while
\[
\int_{|t|}^1 \frac{\cos y}{y} \ell(y/|t|) h(y/|t|) 
\, \dd y 
\sim \int_1^{1/|t|} \frac{\ell(y) h(y) }{y} 
\, \dd y =: L_h(1/|t|).
\]
If $h$ is nonzero then $L_h(x) / \ell(x) \to \infty$ as $x \to \infty$. To
see this write
\[
\liminf_{x \to \infty}
\frac{L_h(x)}{\ell(x)} \geq 
\liminf_{x \to \infty} \int_{\varepsilon x}^x \frac{\ell(u)}{\ell(x)}
\frac{h(u)}{u} \dd u \geq \inf h \, \log \varepsilon^{-1},
\]
as $\varepsilon \downarrow 0$ the claim follows. Moreover, $L_h$ is
slowly varying. Indeed, for $\lambda > 1$ fixed
\[
\begin{split}
L_h(\lambda x) - L_h(x) & 
= \int_{x}^{\lambda x} \frac{\ell(u) h(u)}{u} \dd u \\
& \sim \ell(x) \int_x^{\lambda x} \frac{h(u)}{u} \dd u \leq 
\ell (x) \log \lambda \, \sup h.
\end{split}
\]
Since $\ell(x) / L_h(x) \to 0$, we have $L_h(\lambda x)/ L_h(x) \to 1$,
that is, $L_h$ is slowly varying. The same argument shows that $L$ is
slowly varying too. 
(We note that in (2.6.34) in \cite{IL} it is wrongly stated
that $L_h(x) \sim \ell(x) \log x$.)
The bound for the imaginary part
follows from the inequality $L_h(x) \leq C L(x)$. Finally, if $k \equiv 0$
then $L_h \equiv L$.

\smallskip
\textbf{Strict positivity of the real part.}
The following argument works for any $\alpha \in (0,2)$.
Let $a_0> 0$ be a small number, chosen later. Using that
$\sin y > 2 y / \pi$ for $y \in (0,\pi/2)$ we have
\[
\begin{split}
\Re ( 1 -\varphi(t) ) & = \int_0^\infty (1  - \cos tx) \dd F(x) \\
& = \int_0^\infty 2 \sin^2 \, \frac{tx}{2} \, \dd F(x) \\
& \geq 2 \int_{a_0/t}^{\pi/t} \left( \frac{tx}{\pi} \right)^2 \, \dd F(x) \\
& \geq \frac{2}{\pi^2} a_0^2  \left[ \overline F(a_0/t) - \overline F(\pi/t) \right] \\
& \geq t^\alpha \ell(1/t) \frac{2 a_0^2}{\pi^2} 
\left[ \frac{h(a_0/t)}{a_0^\alpha} \frac{\ell(a_0/t)}{\ell(1/t)} - 
\frac{h(\pi/t)}{\pi^\alpha} \frac{\ell(\pi/t)}{\ell(1/t)} \right].
\end{split}
\]
Since $\ell$ is slowly varying $\ell(\lambda /t ) / \ell(1/t) \to 1$ for any $\lambda$,
therefore the expression in the bracket is strictly positive for $a_0 > 0$ small enough.

\subsection{Local limit theorems}

Before the proof of the LLTs
 we collect some important facts on the characteristic
function $\varphi$, which we use later.

\begin{lemma} \label{lemma:phi-prop}
Let $X$ be an integer valued random variable with span 1
such that 
\eqref{eq:df-ass} holds. Let $\varphi(t) = \E e^{\ii t X}$ denote its
 characteristic function. Then
there exist positive numbers $\nu_1$, $\nu_2$, $\nu_3$
such that
\begin{itemize}
\item[(i)]
if $\alpha \in (0,2)$ then
$| \varphi(t) | \leq e^{- \nu_1 |t|^{\alpha} \ell(1/|t|)}$,   
for $t \in [-\pi, \pi]$. 

\item[(ii)]
if $\alpha \in (0,1)$ then
$| ( 1 - \varphi(t) )^{-1} | \leq \nu_2 |t|^{-\alpha}  \ell(1/t)^{-1}$, 
for $t \in [-\pi, \pi]$;

\item[(iii)] 
if $\alpha \in (0,1)$ then
$| \varphi(t+h) - \varphi(t) | \leq \nu_3 |h|^{\alpha} \ell(1/|h|)$, for 
$t \in \R$, $h \in [-1, 1]$, and if 
$\alpha = 1$ then 
$| \varphi(t+h) - \varphi(t) | \leq
\nu_3 |h| L(1/|h|)$.
\end{itemize}

In the nonlattice case (i)--(iii) remain valid and 
(i)--(ii) can be extended to any compact interval.
\end{lemma}

\noindent \textit{Proof} \ 
Using that $\varphi(t) = e^{\Re \log \varphi(t)}$, and 
$\log \varphi(t) \sim \varphi(t) -1$ around zero, 
the first three statements follows from Theorem \ref{thm:phi-asy-gen}
for $|t|$ small.
Possibly changing the constant, we can extend the inequality to the desired interval.

The fourth inequality follows from
\eqref{eq:df-ass} together with a classical argument; see, for 
instance,~\cite[Proof of Lemma 3.3.2]{GL} or Lemma 5 in \cite{Erickson70}.
\qed

\noindent \textit{Proof of Theorem \ref{thm:local-sst}} \ 
Using the inversion formula \eqref{eq:latt-inv} we have
\[
\begin{split}
\P ( S_n = k) 
& = \frac{1}{2 \pi A_n} \int_{-A_n \pi}^{A_n \pi} e^{- \ii t k / A_n}
\varphi(t/A_n)^n \, \dd t.
\end{split}
\]
By the density inversion theorem the limiting density can be written as
\begin{equation} \label{eq:dens-inv}
g_\lambda(x) = \frac{1}{2 \pi} \int_{-\infty}^\infty e^{-\ii t x} 
\psi_\lambda(t) \, \dd t.
\end{equation}
Thus
\[
2 \pi \, | A_n\, \P(S_n = k) - g_{\gamma_n}((k-C_n)/A_n) | 
\leq I_1 + I_2 + I_3 + I_4,
\]
where
\begin{equation} \label{eq:i-s}
\begin{split}
I_1 & = \int_{-K}^K 
\left| e^{-\ii t C_n / A_n} \varphi(t/A_n)^n - \psi_{\gamma_n} (t) 
\right| \dd t \\
I_2 & = \int_{K \leq |t| \leq \varepsilon A_n} |\varphi(t/A_n)|^n
\, \dd t \\
I_3 & = \int_{\varepsilon A_n \leq |t | \leq \pi A_n} 
|\varphi(t/A_n)|^n \, \dd t \\
I_4 & = \int_{|t| > K} |\psi_{\gamma_n}(t) | 
\, \dd t,
\end{split}
\end{equation}
where $K> 0$ is a large constant.

By Theorem 3.1 in \cite{KCs} the merging relation \eqref{eq:merge} holds
if and only if for any $t \in \R$ as $n \to \infty$
\begin{equation} \label{eq:mergechar}
\E e^{\ii t (S_n - C_n)/ A_n } - \E e^{\ii t V_{\gamma_n}} 
= e^{-\ii t C_n / A_n} \varphi(t/A_n)^n - \psi_{\gamma_n}(t) \to 0.
\end{equation}
Moreover, since both $((S_n - C_n)/A_n)_n$ and $(V_{\gamma_n})_n$
are tight, the convergence in \eqref{eq:mergechar} is uniform on any
finite interval $[-K, K]$. Therefore $I_1 \to 0$ as $n \to \infty$ for
any $K > 0$.

To estimate $I_2$ we use Lemma \ref{lemma:phi-prop} (i) together with the 
Potter bounds. 
Using the inverse relation \eqref{eq:AB} we have
\[
\begin{split}
n (t / A_n )^\alpha \ell ( A_n / t) & =
n t^\alpha \frac{\ell( A_n /t )}{\ell( A_n )}
\frac{\ell(A_n)}{ A_n^\alpha} \\
& \sim t^\alpha \frac{\ell( A_n /t )}{\ell( A_n )} 
\geq 2^{-1} t^{\alpha'},
\end{split}
\]
for any $\alpha' \in (0, \alpha)$, where the last inequality follows from
the Potter bounds.
Therefore, for $\varepsilon > 0$ small enough
\[
I_2 \leq 
\int_K^\infty e^{- 2^{-1} \nu_1 \, t^{\alpha'} } \dd t,
\]
which goes to 0 as $K \to \infty$.

Since $X$ is lattice with span $1$
\begin{equation} \label{eq:lattice}
|\varphi(t)| \leq a < 1 \quad \text{for some }  \, a \in (0,1) \ 
\text{ for } \ |t| \in [\varepsilon, \pi].
\end{equation}
Therefore
$ 
I_3 \leq 2 \pi A_n a^{n},
$ 
while $\psi_{\gamma_n}(t)$ is uniformly integrable by (7) in \cite{Cs07}, 
implying that $\lim_{K \to \infty} I_4 = 0$.
\qed

\noindent \textit{Proof of Theorem \ref{thm:llt-stone}} \ 
We only sketch the proof, because the arguments needed to extend Stone's original 
proof to the semistable case are essentially contained in the proof of Theorem 
\ref{thm:local-sst}.

Changing variables and using \eqref{eq-inversion} and \eqref{eq:dens-inv},
the difference
\[
2\pi \left| 
\frac{A_n}{2h} \P ( S_n + Y \in (x- h, x+h] ) - g_{\gamma_n}((x-C_n)/A_n)
\right|
\]
can be bounded exactly as in \eqref{eq:i-s}, with $TA_n$ instead of $\pi A_n$
in $I_3$. Now, $I_1, I_2$, and $I_4$ can be treated the same way as in the
lattice case, while for $I_3$ we use that by the nonlattice condition
$\sup_{|t| \in [\varepsilon, T]} |\varphi(t)| < 1$ for any $\varepsilon > 0$
and $T > 0$. Thus as $n \to \infty$
\begin{equation} \label{eq:stone-U}
\sup_{x \in \R} 
2\pi \left| 
\frac{A_n}{2h} \P ( S_n + Y \in (x- h, x+h] ) - g_{\gamma_n}((x-C_n)/A_n)
\right| \to 0.
\end{equation}

Using that $Y$ concentrates at 0 as $T \to \infty$, one can get rid of the $Y$
above as in \cite{Stone}. For completeness and later use, we include the 
argument here. Let $h > 0$ be fixed, and let $\delta > 0$. Putting 
$h^+ = ( 1 + \delta ) h$ we have by the independence of $Y$ and $S_n$,
\begin{equation}
\label{eq-SnUanyn}
\P ( S_n \in (x-h, x+h] ) \leq \frac{1}{\P( |Y| \leq \delta h)} 
\P  (S_n + Y \in (x-h^+, x + h^+] ).
\end{equation}
Thus
\begin{equation*} 
\begin{split}
& \frac{A_n}{2h}  \P ( S_n \in (x-h, x+h] ) - g_{\gamma_n}((x-C_n)/A_n) \\
& \leq 
\left( \frac{A_n}{2h^+}  \P  (S_n + Y \in (x-h^+, x + h^+] ) - g_{\gamma_n}((x-C_n)/A_n)
\right) \\
& \quad + \frac{A_n}{2h^+}  \P  (S_n + Y \in (x-h^+, x + h^+] )
\left[ \frac{h^+}{h \P(|Y| \leq \delta h)} -1 \right].
\end{split}
\end{equation*}
By \eqref{eq:stone-U} the first summand tends to 0 as $n \to \infty$ 
for any $\delta$ and $T$. 
Using \eqref{eq:stone-U}  again,
and that $\sup_{\lambda > 0, x \in \R} g_\lambda(x) < \infty$, 
\begin{equation*} 
\sup_{x \in \R} \frac{A_n}{2h^+}  \P  (S_n + Y\in (x-h^+, x + h^+] )
< \infty.
\end{equation*} 
Therefore, choosing first $\delta > 0$ small then $T$ large we obtain
\begin{equation} \label{eq:stone-upper}
\limsup_{n \to \infty}
\sup_{x \in \R}
\frac{A_n}{2h}  \P ( S_n \in (x-h, x+h] ) - g_{\gamma_n}((x-C_n)/A_n)  \leq 0.
\end{equation}
For the lower bound, putting $h^- = ( 1-\delta) h$, using also \eqref{eq:stone-upper}
\[
\begin{split}
& \P (S_n + Y \in (x-h^-, x+h^-])  
= \int_\R \P ( S_n + u \in (x-h^-, x+h^-]) j(u) \dd u \\
& \leq \P(S_n \in (x-h, x+h]) \P (|Y| \leq \delta h) + 
2 \sup_{\lambda > 0, x \in \R } g_\lambda(x) \, \frac{2h}{A_n} \,
\P ( |Y| > \delta h).
\end{split}
\]
Therefore, with $C = 4 \sup_{\lambda > 0, x \in \R } g_\lambda(x)$
\[
\begin{split}
& \P(S_n \in (x-h, x+h]) \geq 
\frac{\P (S_n + Y \in (x-h^-, x+h^-])}
{\P ( |Y| \leq \delta h)}
- C h \frac{\P( |Y| > \delta h)}{A_n \P ( |Y| \leq \delta h)}.
\end{split}
\]
Thus
\[
\begin{split}
& \frac{A_n}{2h} \P(S_n \in (x-h, x+h]) - g_{\gamma_n}((x-C_n)/A_n) \\
& \geq  \frac{A_n}{2h^-} \P(S_n + Y\in (x-h^-, x+h^-]) - g_{\gamma_n}((x-C_n)/A_n) \\
& \  + \frac{A_n}{2h^-} \P(S_n + Y \in (x-h^-, x+h^-]) 
\left( \frac{h^{-}}{h \P ( |Y| \leq \delta h)} -1 \right) 
- C \frac{\P ( |Y| > \delta h)}{\P ( |Y| \leq \delta h )}.
\end{split}
\]
Choosing again first $\delta > 0$ small and then $T > 0$ large we obtain
\begin{equation*} 
\liminf_{n \to \infty}
\inf_{x \in \R}
\frac{A_n}{2h}  \P ( S_n \in (x-h, x+h] ) - g_{\gamma_n}((x-C_n)/A_n)  \geq 0,
\end{equation*}
completing the proof.

For later use, we note that 
the argument implies that for any $\varepsilon > 0$
there exists $T > 0$ such that for $n$ large enough
\begin{equation} \label{eq:Sn-U}
\sup_{x \in \R}
A_n |\P( S_n + Y \in (x-h, x+h] ) - \P( S_n \in (x-h,x+h] ) | \leq  \varepsilon.
\end{equation}
\qed

\subsection{Strong renewal theorems}

We need a continuity property of the densities
$g_\lambda(x)$, in $\lambda$. 
Recall the definition of the constant $c > 0$ in \eqref{eq:H}.
In the following result the interval
$[c^{-2}, c]$ could be replaced by any compact interval of $(0,\infty)$.
For our purpose anything larger than $(c^{-1}, 1]$ would suffice.

\begin{lemma} \label{lemma:g-dens}
There exists $\nu_4 > 0$ such that for any $\lambda_1, \lambda_2 \in 
[c^{-2}, c]$
\[
\sup_{x \in \R} | g_{\lambda_1}(x) - g_{\lambda_2}(x) | \leq 
\nu_4 \, | \lambda_1 - \lambda_2|.
\]
Moreover, 
\begin{equation} \label{eq:g'bound}
\sup_{\lambda \in (c^{-1}, 1]} \sup_{x \in \R}
\frac{\partial}{\partial x} g_\lambda(x) <\infty.
\end{equation}
\end{lemma}

\noindent \textit{Proof} \ 
Introduce the notation 
$\psi_\lambda(t) = \E e^{\ii t V_\lambda} = e^{y_\lambda(t)}$. 
By formula (2.6) in \cite{PK1}
\begin{equation} \label{eq:y-scale}
y_\lambda(t) = \lambda y_1 (t/\lambda^{1/\alpha}) - \ii t c_\lambda,
\end{equation}
with
\[
 c_\lambda = \lambda^{(\alpha - 1)/ \alpha } \int_1^{1/\lambda} 
 \left[ \psi_2(s) - \psi_1(s) \right] \dd s,
\]
where 
$\psi_1(s)  = \inf \{ -x : M_L(x) x^{-\alpha} > s \}$,
$\psi_2(s) = \inf \{ -x : M_R(x) x^{-\alpha} > s \}.$
For any $\lambda > 0$ the function $e^{\lambda y_1(t)}$, $t \in \R$, is a 
characteristic function. Let $G(x ; \lambda)$ denote its 
distribution function, 
i.e.~$e^{\lambda y_1(t)} = 
\int_{\R} e^{\ii t x} G( \dd x; \lambda)$. Cs\"org\H{o} \cite{Cs07}
proved that these functions are infinitely many times differentiable
with respect to both variables. Let $g(x; \lambda)$ be the density 
of $G(x; \lambda)$.

Using the density inversion formula and  \eqref{eq:y-scale} we obtain
\begin{equation} \label{eq:g-form1}
\begin{split}
g_\lambda(x) & =  
\frac{1}{2\pi} \int_{-\infty}^\infty e^{-\ii t x } e^{y_\lambda(t)} \dd t \\
& = 
\lambda^{1/\alpha} \frac{1}{2\pi} 
\int_{-\infty}^\infty e^{-\ii s \lambda^{1/\alpha} (x + c_\lambda) }
e^{ \lambda y_1(s)} \dd s \\
& = \lambda^{1/\alpha} g\left( \lambda^{1/\alpha} (x + c_\lambda); \lambda \right).
\end{split}
\end{equation}
By Lemmas 1 and 2 in \cite{Cs07} for each $j,k$
\begin{equation} \label{eq:mixed-ineq}
\sup_{ \lambda \in [c^{-2}, c]} \sup_{x \in \R} 
\left| \frac{\partial^{j+k}}{\partial x^j \partial \lambda^k} G(x; \lambda) \right|
< \infty,
\end{equation}
which implies that for some constant $C > 0$, for any 
$\lambda_1, \lambda_2 \in [c^{-2}, c]$
\begin{equation*} 
| g(x; \lambda_1) - g(x; \lambda_2) | \leq C |\lambda_1 - \lambda_2|.
\end{equation*}
Using \eqref{eq:g-form1}
\begin{equation*} 
\begin{split}
& g_{\lambda_1} (x) - g_{\lambda_2}(x) = 
\lambda_1^{1/\alpha} \left[ 
g\big( \lambda_1^{1/\alpha}(x+ c_{\lambda_1}), \lambda_1 \big)
- g\big( \lambda_1^{1/\alpha}(x+ c_{\lambda_1}), \lambda_2 \big) \right] \\
& \quad + \lambda_1^{1/\alpha} \left[ 
g\big( \lambda_1^{1/\alpha}(x+ c_{\lambda_1}), \lambda_2 \big)
- g\big( \lambda_2^{1/\alpha}(x+ c_{\lambda_2}), \lambda_2 \big) \right] \\
& \quad + \big( \lambda_1^{1/\alpha} - \lambda_2^{1/\alpha} \big) 
g \big( \lambda_2^{1/\alpha} (x + c_{\lambda_2}); \lambda_2 \big).
\end{split}
\end{equation*}
Using \eqref{eq:mixed-ineq} with 
$j=k=1$, $j=2$, $k=0$, and $j=1$, $k=0$ respectively, and for the second term 
using also that $c_\lambda$ is Lipschitz in $\lambda \in [c^{-2}, c]$,
we obtain
\[
 |g_{\lambda_1} (x) - g_{\lambda_2}(x)  | \leq C |\lambda_1 - \lambda_2|,
\]
as claimed.
The uniform boundedness of the derivatives in \eqref{eq:g'bound}
follows simply from \eqref{eq:g-form1} and \eqref{eq:mixed-ineq}.
\qed

\noindent \textit{Proof of Theorem \ref{thm:SRTGL}} \ 
We estimate $u_n$ via~\eqref{eq:un}. This is possible due to 
Theorem~\ref{thm:phi-asy-gen}, which ensures that 
$\Re\int_0^\pi(1-\varphi(t))^{-1} \dd t$ is well defined.
Let $L > 1$ be a large fixed number.
To ease notation, we suppress the $\lfloor \cdot \rfloor$ notation.
Write
\begin{align*}
\pi u_n&=\Re \int_0^{\pi}(1-\varphi(t))^{-1}\, e^{-\ii nt}\, \dd t\\
&=\left(\sum_{k<B(n/L^2)}+\sum_{k=B(n/L^2)}^{B(nL)}+\sum_{k>B(nL)}\right)
\Re \int_0^{\pi}\varphi(t)^{k}\, e^{-\ii nt}\, \dd t\\
&=: I_1+I_2+I_3.
\end{align*}

First, by Lemma~\ref{lemma-riemsum},
\begin{align}
\label{eq:main}
\limsup_{n \to \infty} \left| n^{1-\alpha} \ell(n) I_2 - 
\pi \alpha \int_{L^{-1}}^{L^2}  
g_{\gamma(B(n) x^{-\alpha})} ( x ) x^{-\alpha} 
\, \dd x \right|\le \frac{\pi}{L}.
\end{align}

Next we handle $I_3$. By Theorem \ref{thm:local-sst} for $k$ large enough
\[
 \sup_n \P (S_k = n ) \leq C \, A_k^{-1},
\]
with $C = 1 + \sup_{\gamma, x} g_\gamma(x)$. Therefore, using Karamata's theorem,
the inverse relation \eqref{eq:AB} and Potter's bounds we obtain 
for any $\varepsilon > 0$
\begin{equation} \label{eq:k22}
\begin{split}
I_3 &  \leq  \pi \sum_{k \geq B(nL)} C \, A_k^{-1} \\
& \sim C \pi \frac{\alpha}{1- \alpha} B(nL)^{1- 1/\alpha} \ell_1(B(nL))^{-1} \\
& \leq  C {n^{\alpha -1}} {\ell(n)^{-1}} L^{\alpha + \varepsilon -1}
\end{split}
\end{equation}
Note that the estimate works for $\alpha \in (0,1)$, the assumption $\alpha > 1/2$ 
is not needed at this point.

It remains to estimate $I_1$. We have
\begin{align*}
|I_1|&\leq \left|\sum_{k<B(n/L^2)} \int_0^{L/n} \varphi(t)^{k}\, e^{-\ii n t}\, 
\dd t\right|
+\left| \sum_{k<B(n/L^2)} \int_{L/n}^{\pi}\varphi(t)^{k}\, e^{-\ii nt }\, \dd t\right|\\
& =: |I_1^1|+|I_1^2| = :|I_1^1|+\left|\sum_{k<B(n/L^2)} I_1^{2, k}\right|.
\end{align*}
Clearly, $|I_1^1| \leq B(n/L^2) \cdot L/n$ and using Potter's bounds, 
for any $\alpha'<\alpha$ for $n$ large enough
\begin{equation} \label{eq:I11}
|I_1^1| \leq 
2 n^{\alpha-1} \ell(n)^{-1} L^{-(2\alpha'-1)}.
\end{equation}
Next, similarly to~\cite[Section 3.5]{GL}, note that
\begin{align} \label{eq:I111decomp}
\nonumber I_1^{2, k} 
& =
\frac{1}{2} \Big(\int_{\pi-\pi/n}^{\pi} +\int_{L/n}^{(L+\pi)/n}\Big) 
\varphi(t)^k \, e^{-\ii nt}\, \dd t \\
\nonumber &\quad + \frac{1}{2} 
\int_{(L+\pi)/n}^\pi 
\left(\varphi(t)^k-\varphi(t-\pi/n)^k\right)\, e^{-\ii nt}\, \dd t \\
& =:J_1^k+J_2^k.
\end{align}
Since, $|J_1^k| \leq \pi/n$, for any $\alpha'<\alpha$ for large $n$,
\begin{equation} \label{eq:J1}
\left|\sum_{k< B(n/L^2)} J_1^{k}\right| \leq 
B(n/L^2) \frac{ \pi}{n} \leq 2  n^{\alpha-1} \ell(n)^{-1} L^{- 2 \alpha'}.
\end{equation}
Using Lemma \ref{lemma:phi-prop} (iii)
\begin{align*}
\left|\varphi(t)^k -\varphi(t-\pi/n)^k\right|
& \le \left|\varphi(t) - \varphi(t-\pi/n)\right|  
\sum_{j=0}^{k-1} |\varphi(t)^j| \, 
|\varphi(t-\pi/n)^{k-j-1}|\\
&\leq 2 \nu_3 \pi^\alpha \, n^{-\alpha} \ell(n) k \,
( |\varphi(t-\pi/n)^{k-1}| + |\varphi(t)^{k-1}|).
\end{align*}
Thus,
\begin{equation} \label{eq:J2-1}
\left|\sum_{k<B(n/L^2)} J_2^{k}\right|
\leq  C n^{-\alpha} \ell(n) 
\sum_{k=0}^{B(n/L^2)} k \int_{L/n}^\pi |\varphi(t)|^k\, \dd t.
\end{equation}
Recall that  $\lim_{k\to\infty}\frac{k\ell(A_k)}{(A_k)^{\alpha}}=1$.
Using Lemma \ref{lemma:phi-prop} (i), change of variables $y\to tA_k$,
and Potter's bound we obtain
\begin{equation} \label{eq:int-bound}
\begin{split}
\int_{L/n}^\pi |\varphi(t)|^k \, \dd t
&\leq \int_{L/n}^\pi e^{- \nu_{1} k   t^{\alpha}\ell(1/t)} \, \dd t \\
& \le \frac{1}{A_k} 
\int_{LA_k/n}^{\pi A_k} e^{- \nu_1 y^\alpha\,k\, A_k^{-\alpha} \ell(A_k/y)}
\, \dd y \\
& \le \frac{1}{A_k} \int_{LA_k/n}^{\pi A_k}
e^{- \frac{\nu_1}{2} y^\alpha \ell(A_k / y)\ell(A_k)^{-1}}  \, \dd y  \\
& \le \frac{1}{A_k}\int_0^\infty 
e^{- C (y^{\alpha-\delta} + y^{\alpha+\delta})} \, \dd y 
 \le \frac{C}{A_k},
\end{split}
\end{equation}
for any $\delta>0$ and some $C>0$.
Recall \eqref{eq:AB}.
Substituting the bound \eqref{eq:int-bound} into \eqref{eq:J2-1},
using Karamata's theorem and that $\alpha>1/2$, we have
\begin{equation} \label{eq:J2-2}
\begin{split}
\left|\sum_{k=0}^{B(n/L^2)} J_2^k \right| & 
\leq C n^{-\alpha} \ell(n) \sum_{k=0}^{B(n/L^2)}\frac{k}{A_k} 
\le C  n^{-\alpha}  \ell(n) 
\frac{B(n/L^2)^{2-\frac{1}{\alpha}}}{\ell_1(B(n/L^2))} \\
& \le C  \frac{n^{\alpha -1}}{\ell(n)} L^{2 - 4 \alpha'},
\end{split}
\end{equation}
with $\alpha' \in (1/2, \alpha)$.

It is worth to note that this is the only part in the proof
where we use that $\alpha > 1/2$. Seemingly, in 
\eqref{eq:I11} we also use this fact, but in that argument
we can enlarge the power of $L$ in $B(n/L^2)$ to work for 
smaller $\alpha$.

Putting \eqref{eq:J1} and \eqref{eq:J2-2} together, recalling that 
$\alpha'<\alpha \in (1/2, 1)$
\begin{align*}
| I_1^2 | = 
\left|\sum_{k<B(n/L^2)} I_1^{2,k}\right|\leq  
C n^{\alpha -1} \ell(n)^{-1} L^{-2\alpha'},
\end{align*}
which combined with \eqref{eq:I11} implies that for any $\alpha' < \alpha$
\begin{equation} \label{eq:I1-bound}
| I_1 | \leq C n^{\alpha -1} \ell(n)^{-1} L^{1-2\alpha'}.
\end{equation}

To finish the proof we have to show that 
\begin{equation} \label{eq:g-int}
\int_0^\infty \sup_{\lambda \in (c^{-1}, 1]} g_\lambda(y) y^{-\alpha} \, \dd y 
< \infty.
\end{equation}
This follows from Theorem 1 by Bingham \cite{Bingham2} (see also
Lemma 2 in \cite{KT}).
By \eqref{eq:g-int} we have
\[
\lim_{L \to \infty} 
\left( \int_0^{L^{-1}} + \int_{L^2}^\infty \right) 
g_{\gamma(B(n) x^{-\alpha})}(x) x^{-\alpha} \, \dd x = 0.
\]
Letting $L\to \infty$ we see that the latter limit together with \eqref{eq:main}, 
\eqref{eq:k22}, and \eqref{eq:I1-bound} imply the statement.
\qed

\noindent \textit{Proof of Lemma \ref{lemma-riemsum}} \
With the same notation as in Theorem~\ref{thm:local-sst}, we write
\begin{align*}
& \frac{1}{\pi} \sum_{k=B(n/L^2)}^{B(nL)}  
\Re \int_0^{\pi} \varphi(t)^{k} \, e^{-\ii nt}\, \dd t
 =   \sum_{k=B(n/L^2)}^{B(nL)} \P(S_k=n) \\
& =  \sum_{k=B(n/L^2)}^{B(nL)} \frac{g_{\gamma_k}(n/A_k)}{A_k}
+ \sum_{k=B(n/L^2)}^{B(nL)} \frac{1}{A_k} [ A_k \P(S_k=n)-g_{\gamma_k}(n/A_k) ].
\end{align*}
By Theorem~\ref{thm:local-sst}, recalling that $C_n \equiv 0$ in our case,
for any $\varepsilon > 0$, for $n$ large enough, for all $k\ge B(n/L^2)$ we have
\[
| {A_k} \P(S_k=n)-g_{\gamma_k}(n/A_k)|< \varepsilon.
\]
Hence, using \eqref{eq:An}, Karamata's theorem and Potter's bound, 
for any $\alpha'<\alpha$, similarly as in \eqref{eq:k22}
\begin{equation} \label{eq:s-gamma}
\begin{split}
& \sum_{k=B(n/L^2)}^{B(nL)} 
\frac{1}{A_k} |A_k \P(S_k=n)-g_{\gamma_k}(n/A_k)|
\leq \sum_{k=B(n/L^2)}^{\infty} \frac{\varepsilon}{ A_k} \\
& \leq \frac{2 \alpha \varepsilon}{1- \alpha} n^{\alpha-1}\ell(n)^{-1} \ L^{2-2\alpha'},
\end{split}
\end{equation}
where in the last inequality we also used the inverse relation
$A(B(n)) \sim n$ in \eqref{eq:AB}.
For $n$ large enough and $L$ fixed, we can take $\varepsilon$ so small that
\begin{equation} \label{eq:LLT-1}
\sum_{k = B(n/L^2)}^{B(nL)} \frac{1}{A_k} 
|A_k \P(S_k=n)-g_{\gamma_k}(n/A_k)| \leq  2^{-1} \, n^{\alpha-1}\ell(n)^{-1} L^{-1}.
\end{equation}

Next, we write  $\sum_{k=B(n/L^2)}^{B(nL)}  A_k^{-1} g_{\gamma_k}(n/A_k)$ as a Riemann sum
proceeding as in~\cite[Lemma 2.2.1]{GL} (see also~\cite[Proof of Th. 8.6.6]{BGT}).
More precisely, set $x_k= k\frac{\ell(n)}{n^\alpha}$.  By definition, $A_k$ is the 
asymptotic
inverse of $n \to \frac{n^\alpha}{\ell(n)} = \frac{k}{x_k}$.
Thus
\begin{equation} \label{eq:xk-range}
L^{-2\alpha -\delta} \leq  B(n/L^2) \frac{\ell(n)}{n^\alpha}\le 
x_k\le B(nL)  \frac{\ell(n)}{n^\alpha}
\leq L^{\alpha + \delta}
\end{equation}
with $\delta>0$ arbitrarily small.
Using the uniform convergence theorem and the inverse relation $B(A_n) \sim A(B(n)) \sim 
n$
(as in~\cite[Proof of Th. 8.6.6]{BGT}), 
we have $x_k^{-1/\alpha}\sim \frac{n}{A_k}$ as $k, n\to\infty$, uniformly in the relevant
range of $k, n$. By \eqref{eq:xk-range} this is equivalent to
\begin{equation} \label{eq:xk-Ak}
\lim_{n \to \infty} 
\sup_{B(n/L^2) \leq k \leq B(nL)} \left| x_k^{-1/\alpha} - \frac{n}{A_k} \right| = 0.
\end{equation}
Since $x_{k+1}-x_k=\frac{\ell(n)}{n^\alpha}$ and $k = B(n) x_k$ 
\begin{align*}
& \sum_{k=B(n/L^2)}^{B(nL)}  
\frac{g_{\gamma_k}(n/A_k)}{A_k}
 =\frac{n^\alpha}{n \ell(n)} \sum_{k=B(n/L^2)}^{B(nL)} 
\frac{n}{A_k} g_{\gamma_k}(n/A_k) \frac{\ell(n)}{n^\alpha}\\
& \sim \frac{n^{\alpha-1}}{\ell(n)} 
\sum_{L^{-2\alpha}<x_k<L^\alpha} x_k^{-1/\alpha}
g_{\gamma(x_k B(n))}(x_k^{-1/\alpha})\left( x_{k+1}-x_k\right),
\end{align*}
where in the last line we have used that
by \eqref{eq:xk-Ak} and by \eqref{eq:g'bound}
we have as $n \to \infty$
\[
\sup_{B(n/L^2) \leq k \leq B(nL) }
|g_{\gamma(x_k B(n))} (n/A_k))-g_{\gamma(x_k B(n))}(x_k^{-1/\alpha})|\to 0.
\]
To finish the proof it is enough to show that
\begin{equation} \label{eq:fn}
f_n(x) := x^{-1/\alpha} g_{\gamma(x B(n))}(x^{-1/\alpha})
\end{equation}
is uniformly Lipschitz on $[L^{-2\alpha}, L^\alpha]$. Indeed, for 
uniformly Lipschitz $f_n$ the convergence of the Riemann sums follows, i.e.
\[
\begin{split}
& \sum_{k=B(n/L^2)}^{B(nL)} x_k^{-\frac{1}{\alpha}}
g_{\gamma(x_k B(n))}(x_k^{-\frac{1}{\alpha}})\left( x_{k+1}-x_k\right) 
\\ & 
\sim
\int_{L^{-2\alpha}}^{L^\alpha} x^{-\frac{1}{\alpha}} 
g_{\gamma(B(n) x)} \big(x^{-\frac{1}{\alpha}} \big) \dd x \\
& = \alpha \int_{L^{-1}}^{L^2} 
g_{\gamma(B(n) y^{-\alpha})}(y) y^{-\alpha} \dd y.
\end{split}
\]
This together with \eqref{eq:LLT-1} implies the statement.
\smallskip

Therefore, it only remains to show that the sequence $(f_n)$ in \eqref{eq:fn} is 
uniformly Lipschitz on any compact subset of $(0,\infty)$. Recall \eqref{eq:def-gamma}
and for $x > 0$ large set $b(x)$ to be the unique index for which
$k_{b(x) - 1} < x \leq k_{b(x)}$. Then $\gamma_x = x/k_{b(x)}$. For some
large $M$ fix the interval
$I = [c^{-M}, c^{M}]$, and let $h > 0$ be small enough such that $1 + h c^M \leq 
\sqrt{c}$.
Then $B(n) (x+h) = B(n) x (1 + h/x) \leq B(n) x \sqrt{c}$, which implies that
$b(B(n) (x + h))$ is either $b(B(n)x)$, or $b(B(n)x) +1$. 
Both cases can be handled similarly, we consider only the former. Then
\[
\gamma(B(n) (x+h)) = \frac{B(n) (x+h)}{k_{b(B(n)x)}} =
\gamma(B(n) x ) + h \frac{B(n)}{k_{b(B(n) x)}}.
\]
The factor of $h$ is $O(1)$ since
\[
\frac{B(n)}{k_{b(B(n) x)}} =  x^{-1} \, \frac{B(n)x}{k_{b(B(n) x)}},
\]
where $x \in I$ and the second factor is less than, or equal to 1.
Thus by Lemma~\ref{lemma:g-dens} the result follows. 
\qed

Before proceeding to the proof of Theorem \ref{thm:SRTEr}, we 
prove Lemma \ref{lemma-sumnonar}.

\noindent \textit{Proof of Lemma~\ref{lemma-sumnonar}}\
Recall that $h \in (0,1/2)$ is fixed.
Proceeding as in the proof of  Lemma \ref{lemma-riemsum}, the conclusion follows 
once we show that as $y\to\infty$,
\begin{equation} \label{eq-toshow}
y^{1-\alpha}\ell(y)\sum_{k=B(y/L^2)}^{B(yL)} 
(I_{k,y}-2h) \frac{g_{\gamma_k}(y/A_k)}{A_k}\to 0.
\end{equation}
Let $R_a$ denote the irrational rotation with $-a$, i.e.
\[
R_a: \R / \Z \to [0,1), \ y \mapsto y - a \quad \text{mod} \ 1.
\]
Note that 
\[
I_{k,y}=1_{[0, h) \cup (1-h, 1)} \circ R_a^k (y).
\]
Let $\eps>0$ be arbitrary. Because of the unique ergodicity property of $R_a$
(see, for instance,~\cite[Section 5]{Oxtoby}), 
there exists $N=N_\eps$ such that for any $n \geq N$
\begin{equation} \label{eq-uniferg}
\sup_{m, y} \left| \frac{\sum_{k=m+1}^{m+n} I_{k,y}}{n} - 2h \right|
=\sup_{m, y} \left| \frac{1}{n} \vert\{ 1 \leq j \leq n:
R_a^{j+m}(y) \in [0, h] \cup (1-h,1) \} 
\vert - 2h \right| \le \eps.
\end{equation}

Divide the interval $[B(y/L^2), B(yL)]$ into blocks $[k_j, k_{j+1})$  of
length $N$. Let 
\[
n_y = \left\lfloor \frac{\lfloor B(yL) \rfloor - \lceil B(y/L^2) \rceil}{N} 
\right\rfloor
\]
and define
\begin{equation} \label{eq-bocks}
k_{j}= \lceil B(y/L^2) \rceil + j N, \quad j = 0, 1,2,\ldots, n_{y} -1,
\quad k_{n_y} = \lfloor B(yL) \rfloor +1.
\end{equation}
Then each block $[k_j, k_{j+1})$ has length $N$, except the last one,
which might be longer, but at most of size $2N$.

By Lemma~\ref{lemma:g-dens},
\begin{equation} \label{eq-ggamm}
\lim_{y \to \infty} \sup_{B(y/L^2) \leq k \leq B(yL)}
\left|
g_{\gamma_{k+1}} (y/A_{k+1}) -
g_{\gamma_k}(y/A_k) \right| = 0.
\end{equation}
Thus, for arbitrarily small $\eps_0$ there exists $y$ sufficiently large,
such that for any $j = 0,1,\ldots$, $n_y - 1$
\begin{equation*} \label{eq-in1}
| g_{\gamma_k}(y/A_k)) - g_{\gamma_{k_j}}(y/A_{k_j})| \leq \eps_0
\quad \mbox{ for every } k\in\{k_j,\ldots,k_{j+1}\}.
\end{equation*}
Next, using properties of slowly varying function, we have that
for arbitrarily small $\eps_1$, there exists $y$ large enough,
such that for any $j = 0,1,\ldots, n_y - 1$
\[
\frac{1}{A_{k_j}} -  \frac{1}{A_k} 
\le \left(\eps_1 +\frac{N}{\alpha k_j}\right)
\frac{1}{A_{k_j}} \mbox{ for every } k \in \{k_j,\ldots,k_{j+1}\}.
\]
As $N$ is fixed and $y \to \infty$, for any $\varepsilon_2 > 0$
there exists $y$ large enough such that 
$N/k_j \leq \varepsilon_2$.
Therefore, with 
$\varepsilon_3 = \varepsilon_0 + \varepsilon_1 + \varepsilon_2$,
for every $k \in \{k_j,\ldots,k_{j+1}-1\}$,
\begin{equation}
\label{eq-in2}
\left| 
\frac{g_{\gamma_k}(y/A_k)}{A_k} - 
\frac{g_{\gamma_{k_j}}(y/A_{k_j})}{A_{k_j}} \right|
\leq \frac{\varepsilon_3}{A_{k_j}}.
\end{equation}
Now,
\begin{align*}
& \sum_{k=B(y/L^2)}^{B(yL)} (I_{k,y}-2h) 
\frac{g_{\gamma_k}(y/A_k)}{A_k} \\
& = \sum_{j=0}^{n_y - 1} 
\sum_{k=k_j}^{k_{j+1}-1} (I_{k,y} -2 h) 
\frac{g_{\gamma_{k_j}}(y/A_{k_j})}{A_{k_j}} 
+ \sum_{j=0}^{n_y - 1} 
\sum_{k=k_j}^{k_{j+1}-1} 2 h 
\left(
\frac{g_{\gamma_{k_j}}(y/A_{k_j})}{A_{k_j}} -
\frac{g_{\gamma_k}(y/A_k)}{A_k} 
\right) \\
& = : S_1 + S_2.
\end{align*}

Using~\eqref{eq-uniferg}, \eqref{eq-in2}, and that 
$A_{k_j} \sim A_k$ uniformly in $k \in \{ k_j, \ldots, k_{j+1} \}$,
we have 
\begin{equation} \label{eq:S1-bound}
\begin{split} 
|S_1| & \le \eps \sum_{j=0}^{n_y - 1} 
(k_{j+1}-k_j) \frac{g_{\gamma_{k_j}}(y/A_{k_j})}{A_{k_j}}\\
& \leq \eps \sum_{j=0}^{n_y -1} \sum_{k=k_j}^{k_{j+1}-1} 
\frac{g_{\gamma_k}(y/A_k) + \varepsilon_3}{A_{k}}. 
\end{split}
\end{equation}
While for $S_2$ by \eqref{eq-in2}  we obtain
\begin{equation} \label{eq:S2-bound}
| S_2 | \leq 
\sum_{k=B(y/L^2)}^{B(yL)} 2 h \frac{\varepsilon_3}{A_k}.
\end{equation}
Since $\varepsilon$ and $\varepsilon_3$ are as small as  we want,
\eqref{eq:S1-bound} and \eqref{eq:S2-bound} imply
\[
\lim_{y \to \infty} \frac{ 
\sum_{k=B(y/L^2)}^{B(yL)} (I_{k,y}-2h) 
A_k^{-1}  {g_{\gamma_k}(y/A_k)}}
{ \sum_{k=B(y/L^2)}^{B(yL)} 
A_k^{-1} {g_{\gamma_k}(y/A_k)}} = 0,
\]
thus \eqref{eq-toshow} follows.
\qed

The proof below  goes by and large as the proof of 
Theorem \ref{thm:SRTGL}. In the lattice case we combine Theorem \ref{thm:SRTGL}
with Lemma~\ref{lemma-sumnonar}.
In the nonlattice case, we use Theorem \ref{thm:llt-stone} and
the inversion formula~\eqref{eq-inversion}
used in the proof of Theorem~\ref{thm:llt-stone}, along with 
the approximation equations~\eqref{eq-SnUanyn} and~\eqref{eq:Sn-U}.
At some extent, our strategy  resembles the one in~\cite{Erickson70}
(suitable for the usual stable/regular variation setting),
but we do not use it a such.

\noindent \textit{Proof of Theorem \ref{thm:SRTEr}} \

\textbf{Nonarithmetic, lattice case.}
We continue from~\eqref{eq:u-nonar} and split the sum into 
$I_1, I_2, I_3$ exactly as in the proof
of Theorem \ref{thm:SRTGL}. The terms $I_1$ and $I_3$ are negligible,
which follows exactly as in the proof of Theorem \ref{thm:SRTGL}. The 
asymptotic of the term $I_2$, which gives the exact term, follows from  
Lemma~\ref{lemma-sumnonar}.

\textbf{Nonlattice case.}
We start from 
\[
\begin{split}
& U(y+h)-U(y-h) =\sum_{k=0}^\infty \P ( S_k  \in (y-h, y+h])\\
&  =
\left( \sum_{k < B(y/L^2)} + \sum_{k=B(y/L^2)}^{B(yL)} + \sum_{k > B(yL)}
\right)\P ( S_k  \in (y-h, y+h] ) \\
& = : E_1 + E_2 + E_3.
\end{split}
\]
For $E_2$ and $E_3$, using~\eqref{eq:Sn-U}, \eqref{eq:s-gamma} (choosing $\eps$ small 
enough) 
and~\eqref{eq-inversion},
\begin{align*}
& E_2+E_3 =\left(\sum_{k=B(y/L^2)}^{B(yL)} + 
\sum_{k > B(yL)}\right)\P ( S_k + Y \in (y- h, y+h] ) 
+ O \left( \frac{y^{\alpha -1}}{\ell(y) L} \right) \\
& = \left( \sum_{k=B(y/L^2)}^{B(yL)} + \sum_{k > B(yL)}\right)
\frac{h}{\pi} \int_{-T}^T \frac{\sin th}{th} e^{-\ii ty} 
\varphi(t)^k \left(1- \frac{|t|}{T} \right) \dd t 
+   O \left( \frac{y^{\alpha -1}}{\ell(y) L} \right) \\
& = : I_2 + I_3 + O\left( L^{-1} y^{\alpha -1} \ell(y)^{-1} \right) .
\end{align*}
The terms $I_2$ and $I_3$ can be treated as their analogues in the proof of Theorem 
\ref{thm:SRTGL}  just writing
$x$ instead of $n$ and $T$ instead of $\pi$. We skip the details, and continue with 
$E_1$. 

Using \eqref{eq-SnUanyn} with $h^+=(1+\delta)h$, for $\delta>0$ and also 
\eqref{eq-inversion},
\begin{align*}
E_1 &\le 
\frac{1}{\P (|Y| \leq \delta h )}
\sum_{k < B(y/L^2)}\P ( S_k + Y \in (y- h^+, y+h^+] )\\
&=\frac{h^+}{\P (|Y| \leq \delta h ) \pi} 
\sum_{k < B(y/L^2)}\int_{-T}^T \frac{\sin th^+}{th^+} e^{-\ii ty}
\varphi(t)^k \left( 1 - \frac{|t|}{T} \right)  \dd t \\
& =: \frac{h^+}{\P (|Y| \leq \delta h ) \pi} I_1.
\end{align*}
To ease notation put
\[
\beta(t) = \frac{\sin (th^+)}{th^+} 
\left( 1- \frac{|t|}{T} \right).
\]
Then $\beta$ is uniformly Lipschitz on $[-T,T]$, thus there is a constant $C$ 
for which
\begin{equation} \label{eq:beta}
|\beta( t ) - \beta (t + s) | \leq C \, s \quad \text{for any } t, t+s \in [-T,T].
\end{equation}
Splitting $I_1$ further as in the arithmetic case, let
\[
\begin{split}
I_1 & = \sum_{k < B(y/L^2)} 
\int_{-T}^T \beta(t) \varphi(t)^k e^{-\ii t y} \, \dd t \\
& = \sum_{k < B(y/L^2)} \left( \int_{|t| \leq L/y} + 
\int_{|t| \in (L/y, T)} \right)
\beta(t) \varphi(t)^k e^{-\ii t y} \, \dd t =: I_1^1 + I_1^2.
\end{split}
\]
As in \eqref{eq:I11} we obtain that for any $\alpha' < \alpha$ for $x$ large enough
\begin{equation} \label{eq:I11-c}
|I_1^1 | \leq 2 y^{\alpha -1} \ell(y)^{-1} L^{-(2 \alpha' -1)}.
\end{equation}
To estimate $I_1^2$, as in the arithmetic case (see also the proof of (5.11) in
\cite{Erickson70}) write
\[
\begin{split}
& \int_{L/y}^T  \beta(t) \varphi(t)^k e^{- \ii t y}  \, \dd t
=
\frac{1}{2} \left( \int_{T - \pi/y}^T + \int_{L/y}^{(L+\pi)/y} \right)
\beta(t) \varphi(t)^k e^{-\ii ty} \, \dd t  \\
& \quad + 
\frac{1}{2} \int_{L/y}^{T - \pi/y}e^{-\ii t y}
\left[ \beta(t) \varphi(t)^k - 
\beta(t+ \pi/y) \varphi(t+ \pi/y)^k \right] \dd t.
\end{split}
\]
Using \eqref{eq:beta} and Lemma \ref{lemma:phi-prop} (iii), as in the
arithmetic case we obtain that for any $\alpha' < \alpha$ for $x$ large enough
\[
|I_1^2| \leq C y^{\alpha - 1} \ell(y)^{-1} L^{-2 \alpha'}.
\]
Combining with \eqref{eq:I11-c} we have
\[
\lim_{L \to \infty} \limsup_{y \to \infty} |I_1| 
\frac{y^{1-\alpha}}{\ell(y)} = 0,
\]
proving the statement.
\qed

\subsection{Renewal function asymptotics}

\noindent \textit{Proof of Theorem \ref{thm:Uasy-sem}} \
We first assume that $X$ is integer valued with span 1.
Let $L > 1$ be a fixed large number. Using \eqref{eq:un} 
\[
\begin{split}
\P ( S_k \leq n ) & = 
\sum_{\ell=0}^n \P ( S_k = \ell) \\
& = \frac{1}{2 \pi} \int_{-\pi}^\pi 
\sum_{\ell = 0}^n e^{-\ii \ell t} \varphi(t)^k \, \dd t 
\\
& = \frac{1}{2 \pi} \int_{-\pi}^\pi 
\frac{1 - e^{- \ii (n+1) t}}{1- e^{-\ii t}} 
\varphi(t)^k \, \dd t,
\end{split}
\] 
thus
\begin{equation*} 
U(n) = \sum_{k = 0}^\infty \P( S_k \leq n ) = 
\frac{1}{2 \pi} \int_{-\pi}^\pi 
\frac{1 - e^{- \ii (n+1) t}}{1- e^{-\ii t}}  \frac{1}{1- \varphi(t)} \dd t.
\end{equation*}
First we show that the main contribution in $U(n)$ comes from the integral on
$[(nL)^{-1}, L/n]$. Indeed, for $|t| \geq L/n$, using Lemma \ref{lemma:phi-prop}
(ii)
\begin{equation} \label{eq:larget}
\begin{split}
\left| \int_{\frac{L}{n} \leq |t| \leq \pi} 
\frac{1 - e^{- \ii (n+1) t}}{1- e^{-\ii t}}  \frac{1}{1- \varphi(t)} \dd t
\right| 
& \leq C \int_{L/n}^\pi \frac{1}{t} t^{-\alpha} \ell(1/t)^{-1} \dd t \\
& \leq C \frac{n^\alpha}{\ell(n)} L^{-\alpha}, 
\end{split}
\end{equation}
while for $|t| \leq 1/(nL)$
\begin{equation} \label{eq:smallt}
\begin{split}
\left| \int_{|t| \leq (nL)^{-1}} 
\frac{1 - e^{- \ii (n+1) t}}{1- e^{-\ii t}}  \frac{1}{1- \varphi(t)} \dd t
\right|
& \leq C \int_0^{(nL)^{-1}} n  t^{-\alpha} \ell(1/t)^{-1} \dd t \\
& \leq C \frac{n^\alpha}{\ell(n)} L^{\alpha - 1}. 
\end{split}
\end{equation}
Therefore we need to consider the integral on $[(nL)^{-1}, L/n]$. Write
\begin{equation*} 
\begin{split}
&\int_{\frac{1}{Ln} \leq |t| \leq \frac{L}{n}}
\frac{1 - e^{- \ii (n+1) t}}{1- e^{-\ii t}}
\left[ 
\sum_{k=0}^{ B(n/\sqrt{L})} + 
\sum_{k= B(n/\sqrt{L})}^{B(nL^2)} + \sum_{ k > B(nL^2)}
\right]
\varphi(t)^k \, \dd t \\
&
= : I_1 + I_2 + I_3.
\end{split}
\end{equation*}
The arguments  below are somewhat similar to the ones in the proof of Theorem 
\ref{thm:SRTGL}, but simplified.
For the first term for $n$ large enough
\begin{equation*} 
|I_1 | \leq C \int_{1/(Ln)}^{L/n} \frac{1}{t} B(n/\sqrt{L}) \dd t
\leq C B(n/\sqrt{L}) \log L
\leq C \frac{n^\alpha}{\ell(n)} L^{-\alpha/2} \log L,
\end{equation*}
while for the third using Lemma \ref{lemma:phi-prop} (i) and (ii) 
and the uniform convergence
theorem for slowly varying functions we obtain for $n$ large enough
\begin{equation*} 
\begin{split}
| I_3 | & \leq C \int_{ \frac{1}{Ln}}^{\frac{L}{n}} \frac{1}{t} 
e^{- \nu_1 t^\alpha \ell(1/t) B(nL^2)} t^{-\alpha} \ell(1/t)^{-1} \dd t \\
& \leq  C \frac{1}{\ell(n)} \int_{ \frac{1}{Ln}}^{\frac{L}{n}}
 t^{-\alpha - 1}
e^{-\frac{\nu_1}{2} (nL^2t)^\alpha} \dd t \\
& \leq C \frac{1}{\ell(n)} \int_{\frac{1}{Ln}}^1 
t^{-\alpha - 1} \, \dd t \, 
e^{-\frac{\nu_1}{2} L^\alpha} \\
& \leq  C \frac{n^\alpha}{\ell(n)} L^{\alpha} 
e^{-\frac{\nu_1}{2} L^\alpha}.
\end{split}
\end{equation*}

It remains to estimate $I_2$. 
For $ B(n/\sqrt{L}) \leq k \leq B(nL^2)$ uniformly in $k$ as
$n \to \infty$ we have
\[
\begin{split}
\int_{\frac{1}{Ln} \leq |t| \leq \frac{L}{n}}
\frac{1 - e^{- \ii (n+1) t}}{1- e^{-\ii t}}
\varphi(t)^k \, \dd t 
& \sim \int_{\frac{1}{Ln} \leq |t| \leq \frac{L}{n}}
\frac{1 - e^{- \ii (n+1) t}}{\ii t}
\varphi(t)^k \, \dd t 
=: I_2^k.
\end{split}
\]
Changing variables and using the usual inversion formula
for characteristic functions
\begin{equation*} 
\begin{split}
I_2^k & =  
\int_{\frac{A_k}{Ln} \leq |u| \leq \frac{LA_k}{n}}
\frac{1 - e^{- \ii \frac{n+1}{A_k} u }}{\ii u}
\varphi(u/A_k)^k \, \dd u \\
& = \int_{-\infty}^\infty
\frac{1 - e^{- \ii \frac{n+1}{A_k} u}}{\ii u} 
\psi_{\gamma_k}(u) \, \dd u \\
& \quad
- \left( \int_{|u| \leq \frac{A_k}{Ln}} + 
\int_{|u| \geq \frac{L A_k}{n}} \right)
\frac{1 - e^{- \ii \frac{n+1}{A_k} u}}{\ii u} \psi_{\gamma_k}(u) \dd u \\
& \quad + \int_{\frac{A_k}{Ln} \leq |u| \leq \frac{LA_k}{n}}
\frac{1 - e^{- \ii \frac{n+1}{A_k} u}}{\ii u}
\left(  \varphi(u/A_k)^k - \psi_{\gamma_k}(u)  \right)
\dd u \\
& = G_{\gamma_k} \left( \frac{n+1}{A_k} \right) 
- J_1^k - J_2^k + J_3^k. 
\end{split}
\end{equation*}
Since $A_k/n$ ranges from $L^{-1/2}$ to $L^2$, 
it can be shown as in \eqref{eq:xk-range} that 
for any fixed $L$ the interval 
$[A_k / (Ln), L A_k / n]$ for $B(n/\sqrt{L}) \leq k \leq B(nL^2)$ is 
bounded away both from 0 and from $\infty$ uniformly in $k$. 
The merging relation implies that \eqref{eq:mergechar} holds, therefore
\begin{equation} \label{eq:U-J3}
\lim_{n \to \infty} \sup_{B(n/\sqrt{L}) \leq k \leq B(nL^2)} J_3^k = 0.
\end{equation}
Since $G_\gamma$ has a density $g_\gamma$, the characteristic function
$\psi_\gamma$ is integrable, and as $n \to \infty$
\[
\frac{L A(B(n/\sqrt{L}))}{n} \sim \sqrt{L},
\]
which tends to $\infty$ as $L \to \infty$, we have that
\begin{equation} \label{eq:U-J2}
\lim_{L \to \infty}
\limsup_{n \to \infty} \sup_{B(n/\sqrt{L}) \leq k \leq B(nL^2)} J_2^k = 0.
\end{equation}
Finally, for $J_1^k$ note that for $L$ large 
\[
\left| 1 - e^{- \ii \frac{n+1}{A_k} u} \right| \leq 
2 \frac{n+1}{A_k} |u|
\]
whenever $|u| \leq A_k/(Ln)$. Thus
\begin{equation} \label{eq:U-J1-1}
| J_1^k | \leq 2 \frac{n+1}{A_k} \frac{A_k}{Ln} \leq \frac{3}{L}.
\end{equation}
Putting together \eqref{eq:U-J3}, \eqref{eq:U-J2}, and \eqref{eq:U-J1-1},
we obtain that for any $\varepsilon > 0$ we can choose $L$ large enough such
that for $n$ large enough
\begin{equation} \label{eq:U-Ik}
\sup_{B(n/\sqrt{L}) \leq k \leq B(nL^2)} 
\left|I_2^k - G_{\gamma_k}\left( \frac{n+1}{A_k} \right) \right|
\leq \varepsilon.
\end{equation} 

Finally, as in the proof of Lemma \ref{lemma-riemsum} we obtain that
\[
\begin{split}
\sum_{k=B(n/\sqrt{L})}^{B(nL^2)} 
G_{\gamma_k} \left( \frac{n+1}{A_k} \right) 
& \sim 
\frac{n^\alpha}{\ell(n)} \int_{L^{-\alpha/2}}^{L^{2\alpha}} 
G_{\gamma(B(n) x)}(x^{-1/\alpha}) \, \dd x \\
& = \frac{n^\alpha}{\ell(n)} \alpha \int_{L^{-2}}^{\sqrt{L}}
G_{\gamma(B(n) u^{-\alpha})} ( u) u^{-\alpha -1} \, \dd u.
\end{split}
\]
This completes the proof in the arithmetic case.
\bigskip

The nonarithmetic case is similar.
The only difference in this case is the expression of the inversion formula.
As in \eqref{eq-inversion} (with $Y$ defined in \eqref{eq:def-Y}),
\[
\P ( S_k +Y\leq y )  
= \frac{1}{2 \pi} \int_{-T}^T \frac{1 - e^{- \ii y t}}{\ii t} 
\varphi(t)^k (1- |t|/T) \, \dd t
\] 
which gives
\begin{equation*} 
\sum_{k = 0}^\infty \P ( S_k +Y\leq y ) = 
\frac{1}{2 \pi} \int_{-T}^T 
\frac{1 - e^{- \ii y t}}{\ii t}  \frac{1}{1- \varphi(t)}  
(1- |t|/T) \, \dd t.
\end{equation*}
Proceeding as in the argument above in the integer valued case 
with $y$ instead of $n$, $T$ instead of $\pi$ 
and  $\ii t$ instead of $1- e^{-\ii t}$, we obtain
the analogues of \eqref{eq:larget}, \eqref{eq:smallt}, and 
\eqref{eq:U-Ik}.
Putting these together,
\begin{equation*}
\lim_{y\to\infty}
\Big| y^{-\alpha}\ell(y) \sum_{k = 0}^\infty \P ( S_k +Y\leq y ) 
- \alpha \int_{0}^{\infty}
G_{\gamma(B(y) x^{-\alpha})} ( x) x^{-\alpha -1} 
\, \dd x \Big| =0.
\end{equation*}
To complete, we need to get rid of $Y$ in the above equation. 
This can be done using \eqref{eq:Sn-U}.
\qed

\bigskip

\noindent \textbf{Acknowledgement.}
We are thankful to Vilmos Totik for showing us a simpler 
proof of the strict positivity
of the real part in Theorem \ref{thm:phi-asy-gen} and
to the anonymous referee for the remarks and suggestions, in
particular for pointing out reference \cite{Uchi}.


\end{document}